\documentclass[%
 aip,
 pof,%
 amsmath,amssymb,
preprint,%
numerical,%
]{revtex4-1}

\usepackage{graphicx}
\usepackage{subfig}

\usepackage{dcolumn}
\usepackage{bm}

\usepackage{amsthm}
\newcommand\sign{\mathop{\rm sign}}
\newtheorem{theorem}{Theorem}[section]

\newtheorem{remark}{Remark}[section]
\usepackage{xcolor} 
\draft 
\begin{document}


\title{A Floating Cylinder on An Unbounded Bath} 



\author{Hanzhe Chen}
\email{h248chen@uwaterloo.ca}
\affiliation{Department of Applied Mathematics, University of Waterloo}


\author{David Siegel}
\email{dsiegel@uwaterloo.ca}
\affiliation{Department of Applied Mathematics, University of Waterloo}

\date{\today}

\begin{abstract}
In this paper, we reconsider a circular cylinder horizontally floating on an unbounded reservoir in a gravitational field directed downwards, which was studied by Bhatnargar and Finn\cite{MR2259294} in 2006. We follow their approach but with some modifications. We establish the relation between the total energy $E_T$ relative to the undisturbed state and the total force $F_T$, that is, $F_T = -\frac{dE_T}{dh}$, where $h$ is the height of the center of the cylinder relative to the undisturbed fluid level. There is a monotone relation between $h$ and the wetting angle $\phi_0$. We study the number of equilibria, the floating configurations and their stability for all parameter values. We find that the system admits at most two equilibrium points for arbitrary contact angle $\gamma$, the one with smaller $\phi_0$ is  stable and the one with larger $\phi_0$ is unstable. The initial model has a limitation that the fluid interfaces may intersect. We show that the stable equilibrium point never lies in the intersection region, while the unstable equilibrium point may lie in the intersection region. 
\end{abstract}


\maketitle 


\section{Introduction}

This study of a circular cylinder horizontally floating on an unbounded bath is motivated by the ground breaking paper of Bhatnagar and Finn \cite{MR2259294}. They considered equilibrium configurations and their stability from both the energy and the total force points of view. They gave a surprising example with two distinct equilibrium configurations. We are interested in investigating the number of equilibrium configuration and their stability for all values of the parameters. We will follow Bhatnagar and Finn's approach but with some modifications. In Sec.~\ref{fatea}, we consider the total energy relative to the undisturbed fluid $E_T$ and establish the relation between it, and the total force $F_T$:
\begin{equation}
-\frac{dE_T}{d h} = F_T,
\label{detdh}
\end{equation}
where $h$ is the height of center. This relation provides a convenient way to analyze the stability of equilibrium configurations based only on $F_T$. In Sec.~\ref{stabe}, we study the behavior of the total force curve and conclude there are at most two equilibrium configurations. The initial model has a limitation due to the possible intersection of fluid interfaces that is not physically realizable. In Sec.~\ref{noe}, we discuss the intersection condition of the fluid interfaces. Taking this into consideration, we determine the number of equilibria and their stability in the physically realizable cases for typical contact angles. 

The validity of ``Young's diagram'' has been widely discussed in recent literature. Bhatnagar and Finn\cite{MR2259294} assert that the surface tension force $F_\sigma$ acts only along the fluid interface. The relation (\ref{detdh}) we obtained implicitly supports their assertion. Related discussions of Young's diagram and Finn's counterexample can also be found in (Refs.~\onlinecite{MR2259294,youngpara,MR2964743,youngdiscussion1, MR2259293,youngdiscussion2}). 

Archimedes' principle is not in general correct when the air/liquid interface is not flat due to the presence of surface tension. McCuan and Treinen study Archimedes' principle and capillarity in Ref.~\onlinecite{MR3095116}. In Appendix \ref{Apx2}, we discuss the validity of Archimedes' principle for the floating cylinder problem in presence of the surface tension.

\subsection{Fluid Interface and Configuration}
Suppose an infinite reservoir of fluid has its interface with the air at the zero level. Introduce an infinite circular cylinder of radius $a$ floating horizontally on the infinite reservoir and assume the free fluid level is unchanged. 
If we admit the presence of surface tension, the fluid will be lifted up or pushed down to the fluid height $u$. Consider the fluid interface on the right. The inclination angle $\psi$ is measured counterclockwise from the positive horizontal direction. When $u>0$, $\psi$ ranges from the top ($\psi = -\pi$) to the free fluid level $\psi = 0$. Oppositely, $\psi$ ranges from $0$ to $\pi$ when $u<0$ (see Fig.~\ref{inclination}).

\begin{figure*}[h]
     \centering
     \subfloat[]{\includegraphics[scale = 0.3]{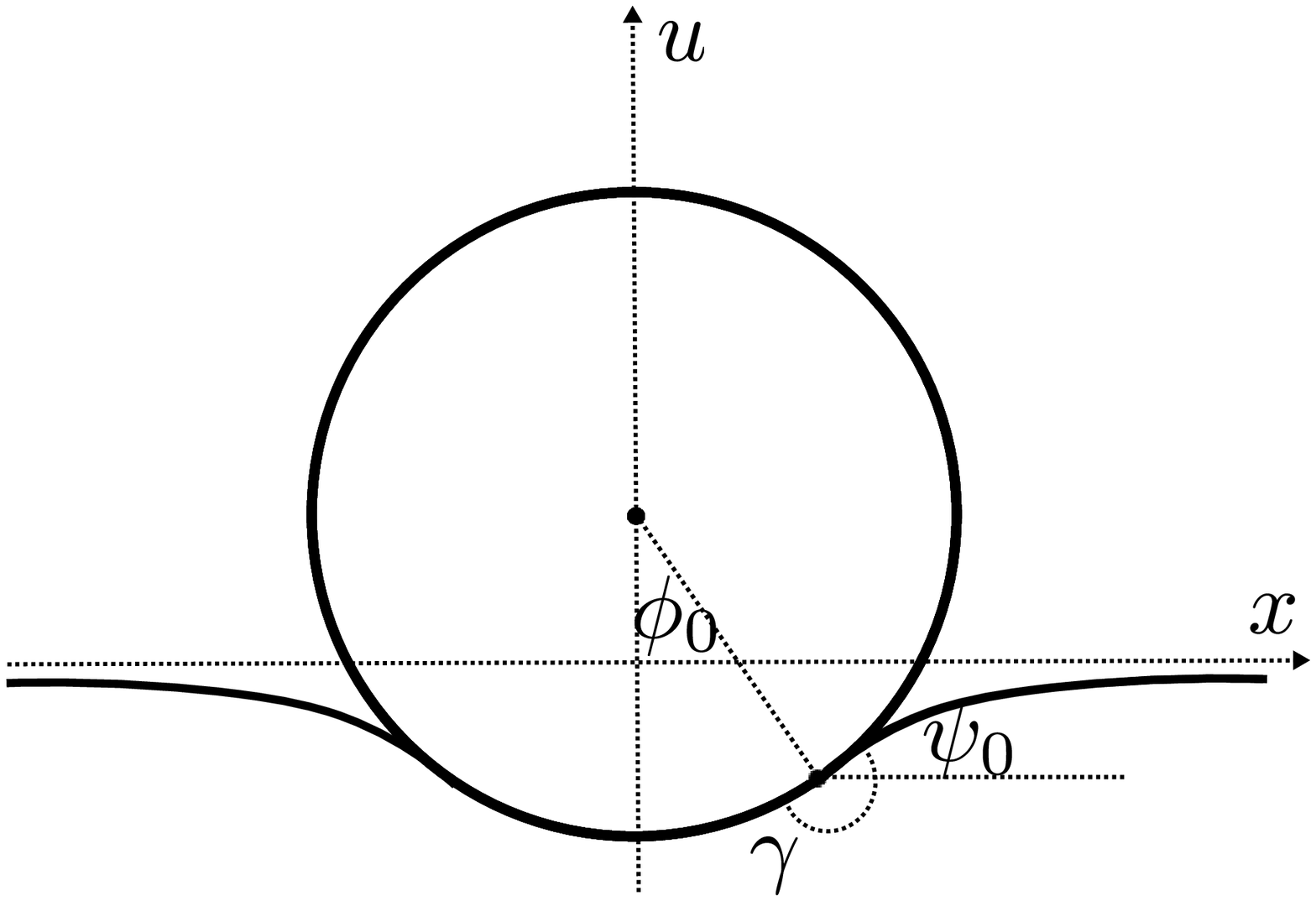}\label{configuration}}
     \hspace{1cm}
     \subfloat[]{\includegraphics[scale = 0.3]{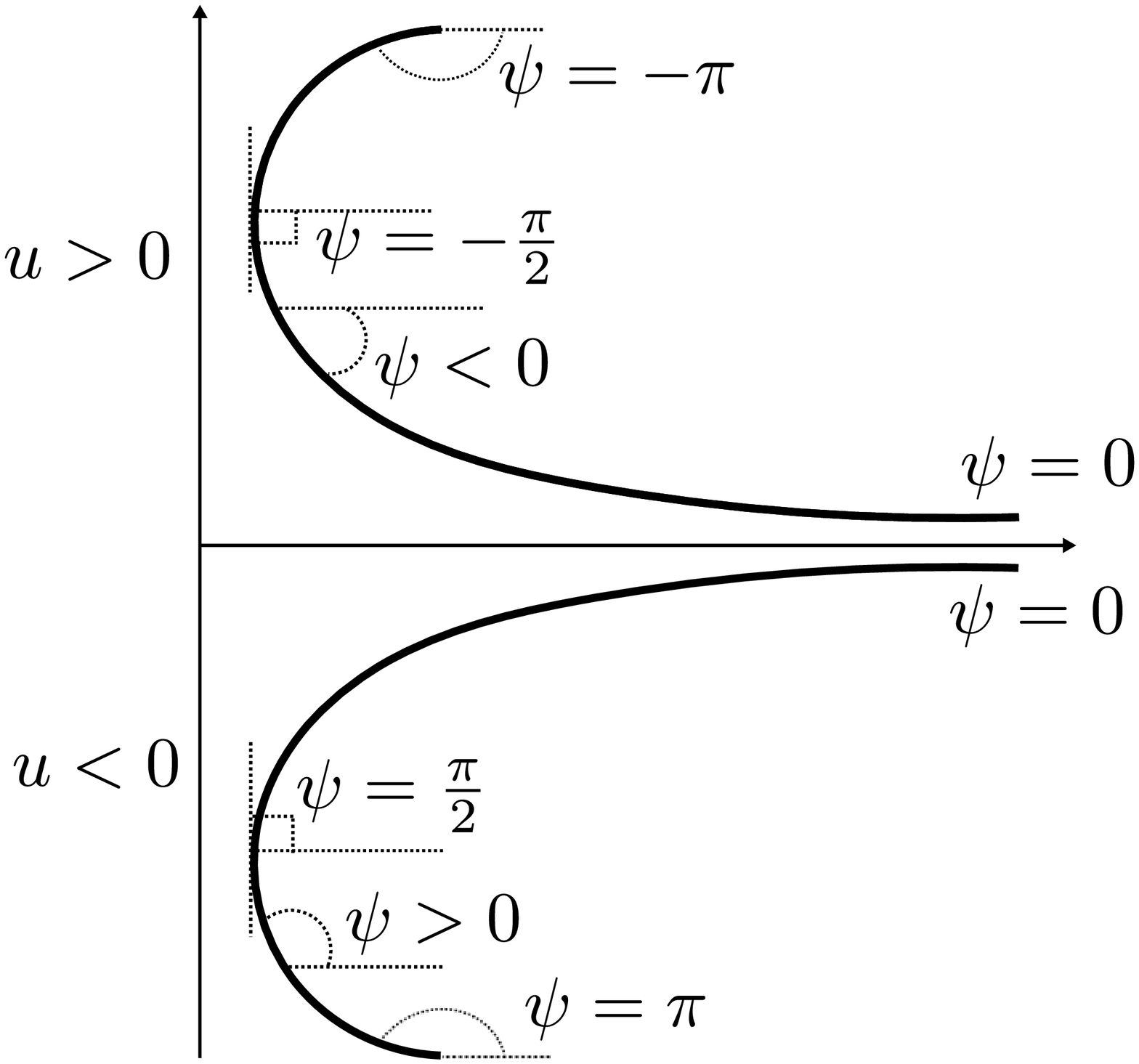}\label{inclination}}
     \caption{(a) The cross-sectional configuration of the horizontal cylinder lying on the liquid and (b) the measurement of the inclination angle $\psi$ for both positive and negative fluid heights.}
\end{figure*}

Assume that all the fluid, the air and the cylinder are homogeneous. Considering a unit length along the cylinder, our ideal model turns into a two-dimensional problem. Viewing the cross section in Fig.~\ref{configuration}, we set the center of the cylinder on the vertical axis. At the contact point between the fluid and the cylinder, we define the contact angle $\gamma$, the inclination angle $\psi_0$ and the wetting angle $\phi_0$. Immediately, we obtain the geometric constraint:
\begin{equation}
\psi_0=\phi_0+\gamma-\pi.
\label{gconstraint}
\end{equation}

\subsection{The Capillary Equation}\label{thecapequation}
Since the configuration is symmetric about the vertical axis, it suffices to look at the fluid interface on positive side, $x>0$. Geometrically, the curvature of the fluid interface $\frac{d\psi}{ds}$ is proportional to the fluid height $u$ with constant $\kappa$, that is the one-dimensional capillary equation
\begin{equation}
\frac{d\psi}{ds} = \kappa u,
\label{capeqn}
\end{equation}
where $s$ is the arc length, $\kappa = \frac{\rho g}{\sigma}$ is known as the capillary constant,  $\sigma$ is surface tension along the fluid interface, $\rho$ is the density difference of the fluid and the air, and $g$ is the acceleration due to gravity. 

We assume the fluid height $u$ goes to zero asymptotically as $\psi \rightarrow 0$. That is
\begin{equation}
\lim_{\psi \rightarrow 0} u(\psi) = 0.
\label{bc}
\end{equation}

The capillary Eq.~(\ref{capeqn}) with the boundary condition (\ref{bc}) admits a unique solution $u(\psi)$ and $x(\psi)$ up to translation. This solution is classically known and can be traced back to Laplace and Euler. 

Finn and Bhatnagar\cite{MR2259294} have given the solution in terms of $\psi$. We modify the solution to treat $u>0$ and $u<0$ simultaneously and to be consistent with our notation.  
\begin{eqnarray}
u(\psi) &=& -\frac{2}{\sqrt{\kappa}} \sin\frac{\psi}{2},\label{solnu}\\
x(\psi) &=& -\frac{1}{\sqrt{\kappa}}\left(2\cos\frac{\psi}{2}+\ln\left\vert\tan\frac{\psi}{4}\right\vert-2\cos\frac{\psi_0}{2}-\ln\left\vert\tan\frac{\psi_0}{4}\right\vert\right)+a\sin\phi_0.\label{solnx}
\end{eqnarray}

At the contact point, the horizontal distance is $x_0 = a\sin\phi_0$, the fluid height at $x_0$ is 
\begin{equation}
u_0=u(\psi_0)=-\frac{2}{\sqrt{\kappa}}\sin{\frac{\psi_0}{2}}. 
\end{equation}

The height of the center $h = a\cos\phi_0+u_0$, therefore
\begin{equation}
h = a\cos\phi_0-\frac{2}{\sqrt{\kappa}}\sin\frac{\psi_0}{2}.
\end{equation}

\begin{figure}[h]
	\includegraphics[scale = 0.3]{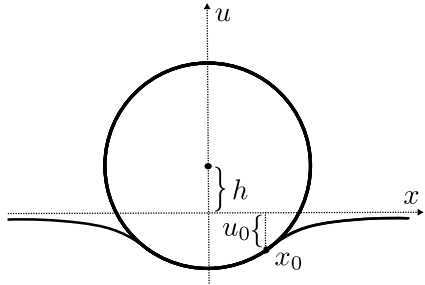}
	\caption{The coordinates of the contact point $x_0$ and $u_0$, and the height of the center $h$.}
\end{figure}

\section{Force Analysis and Total Energy Approach}\label{fatea}
\subsection{Derivation of the Total Energy $E_T$}\label{detoen}
In this section, following the method of Gauss \cite{MR816345}, we determine all the potential energies of the floating cylinder system. Because of the unboundedness of the fluids, we consider the relative energy to avoid the confusion of infinite energy. The types of energies will be expressed explicitly in terms of the inclination angle at the contact point $\psi_0$ and the wetting angle $\phi_0$. 

We have the following four types of energy:
\begin{enumerate}
\item The body potential energy which is relative to the free fluid level can be expressed as $E_G = mgh$, where $h = a\cos\phi_0-\frac{2}{\sqrt{\kappa}}\sin\frac{\psi_0}{2}$. $E_G$ is a function in terms of $\psi_0$ and $\phi_0$:  
\begin{equation}
E_G(\psi_0, \phi_0) = mg\left(a\cos\phi_0-\frac{2}{\sqrt{\kappa}}\sin\frac{\psi_0}{2}\right).
\end{equation}

\item The wetting energy $E_W= -\beta\sigma\left\vert\Sigma\right\vert$, where the wetting area per unit length is denoted by $\left\vert\Sigma\right\vert=2a\phi_0$. With the relative adhesion coefficient $\beta$, $E_W$ only depends on $\phi_0$:
\begin{equation}
E_W(\phi_0) = -2\beta\sigma a \phi_0,
\end{equation}
where $\beta$ can be shown to be equal to $\cos\gamma$, see Ref.~\onlinecite{hanzhe,MR816345}.

\item Surface tension can be interpreted as energy per area. To avoid infinite energy, we define the surface energy $E_{\sigma}$, a relative energy compared with the surface energy of undisturbed fluid surface (see Fig.~\ref{surfacetension}). When the fluid interfaces are graphs, it has the form
\begin{equation}
E_{\sigma} = 2\sigma\lim_{x_1 \rightarrow \infty}\left[\int_{x_0}^{x_1}\sqrt{1+\left(\frac{du}{dx}\right)^2}\,dx - \int_{0}^{x_1}\,dx\right].
\end{equation}

\begin{figure}[h]
	\includegraphics[scale = 0.3]{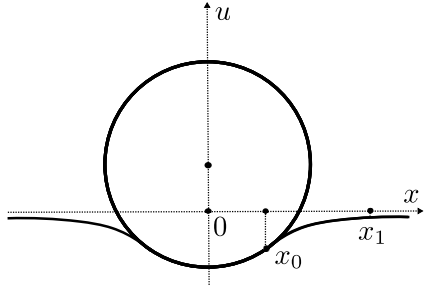}
	\caption{Computation of $E_{\sigma}$ when the fluid interface is a graph.}
	\label{surfacetension}
\end{figure}

The fluid interface may also be a non-graph. The details of computing $E_\sigma$ in both cases are in Appendix~\ref{Apx1}. $E_\sigma$ is shown below: 
\begin{equation}
E_{\sigma}(\psi_0, \phi_0) = \frac{4\sigma}{\sqrt{\kappa}}(1-\cos\frac{\psi_0}{2}) -2\sigma a \sin\phi_0.
\end{equation}

\item $E_F$ is the potential energy of fluids which are lifted or displaced comparing to the free fluid level. The case when the fluid interface and the cross section of the wetted region are both graphs is shown in Fig.~\ref{fluidpo}.

\begin{figure}[h]
	\includegraphics[scale = 0.3]{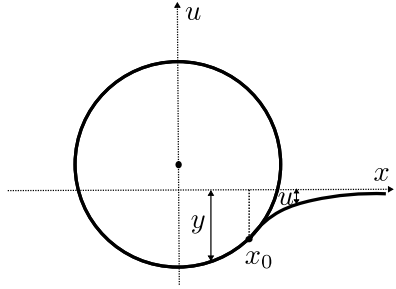}
	\caption{Computation of $E_F$.}
	\label{fluidpo}
\end{figure}

It is helpful to break $E_F$ into two parts $E_{F1}$ and $E_{F2}$,
\begin{equation}
E_F = \underbrace{2\rho g \int_{0}^{x_0} \frac{y^2}{2}\,dx}_{E_{F1}}+ \underbrace{2 \rho g\int_{x_0}^{\infty}\frac{u^2}{2}\,dx}_{E_{F2}},
\label{EF1EF2}
\end{equation}
where $y$ is the vertical height from the free fluid level to the bottom of the cylinder (see Fig.~\ref{fluidpo}). For computation details, see Appendix~\ref{Apx1}. Thus, 
\begin{eqnarray}
E_F(\psi_0,\phi_0) &=& E_{F_1}+E_{F_2} \nonumber\\
&=& -\frac{4\sigma}{3\sqrt{\kappa}}\left(1-2\cos\frac{\psi_0}{2}+\cos\frac{\psi_0}{2}\cos\psi_0\right) \nonumber +\frac{1}{12}\rho ga^3\sin3\phi_0\\
&& -\rho ga^3\phi_0\cos\phi_0+\frac{3}{4}\rho g a^3\sin\phi_0-a^2\sqrt{\sigma \rho g}\sin\frac{\psi_0}{2}\sin2\phi_0 \nonumber\\
&& +2a^2\sqrt{\sigma \rho g}\phi_0\sin\frac{\psi_0}{2}+4\sigma a \sin^2\frac{\psi_0}{2}\sin\phi_0.
\label{toeqn}
\end{eqnarray}
The same expression of $E_F$ holds for the general case. 
\end{enumerate}

The total energy $E_T$ can be expressed of the sum of the above four energies. 
\begin{equation}
E_T = E_G + E_W + E_{\sigma}+E_F.
\end{equation}

The full expression of $E_T(\psi_0,\phi_0)$ is
\begin{eqnarray}
E_T(\psi_0,\phi_0) &=& mg\left(a\cos\phi_0-\frac{2}{\sqrt{\kappa}}\sin\frac{\psi_0}{2}\right)-2\beta \sigma a \phi_0+\frac{4\sigma}{\sqrt{\kappa}}\left(1-\cos\frac{\psi_0}{2}\right)-2\sigma a \sin\phi_0\nonumber\\
&& -\frac{4\sigma}{3\sqrt{\kappa}}\left(1-2\cos\frac{\psi_0}{2}+\cos\frac{\psi_0}{2}\cos\psi_0\right)+\frac{1}{12}\rho ga^3\sin3\phi_0 -\rho ga^3\phi_0\cos\phi_0\nonumber \\
&& +\frac{3}{4}\rho g a^3\sin\phi_0 -a^2\sqrt{\sigma \rho g}\sin\frac{\psi_0}{2}\sin2\phi_0+2a^2\sqrt{\sigma \rho g}\phi_0\sin\frac{\psi_0}{2}\nonumber \\
&& +4\sigma a\sin^2\frac{\psi_0}{2}\sin\phi_0.
\end{eqnarray}

With $\kappa = \frac{\rho g}{\sigma}$, $\beta = \cos\gamma$ and the geometric constraint $\psi_0 = \phi_0+\gamma -\pi$, after some calculation, the total energy $E_T(\psi_0, \phi_0)$ can be converted to
\begin{eqnarray}
E_T(\phi_0)&=& mg\left[a\cos\phi_0+2\sqrt{\frac{\sigma}{\rho g}}\cos\left(\frac{\phi_0+\gamma}{2}\right)\right]-2 \sigma a\phi_0\cos\gamma+\frac{8}{3}\sigma \sqrt{\frac{\sigma}{\rho g}}\left[1-\sin^3\left(\frac{\phi_0+\gamma}{2}\right)\right] \nonumber\\
&&+2\sigma a \sin\phi_0 \cos\left(\phi_0+\gamma\right)+\frac{1}{12}\rho ga^3\sin3\phi_0-\rho ga^3\phi_0\cos\phi_0+\frac{3}{4}\rho g a^3\sin\phi_0 \nonumber \\
&&+a^2\sqrt{\sigma \rho g}\cos\left(\frac{\phi_0+\gamma}{2}\right)\sin2\phi_0-2a^2\sqrt{\sigma \rho g}\phi_0\cos\left(\frac{\phi_0+\gamma}{2}\right).
\label{totalenergy}
\end{eqnarray}

\subsection{Analysis of the Forces}\label{analysisforces}

By symmetry, the surface tension forces in the horizontal direction cancel so that the net force in the horizontal direction is zero. Thus we only need to consider forces in the vertical direction. Bhatnagar and Finn\cite{MR2259294}
 give an analysis of the forces. We suppose upward is the positive direction and modify the expression of the forces as follows.
 
\begin{figure}[h]
	\includegraphics[scale = 0.3]{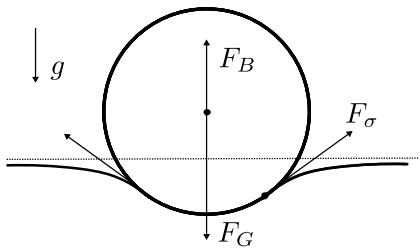}
	\caption{Gravitational, buoyant and surface tension forces.}
	\label{toforce}
\end{figure}

\begin{enumerate}

\item The gravitational force $F_G$ is caused by the downward pointed gravitational field $g$ and the mass of a unit length $m$. $F_G$ can be expressed as
\begin{equation}
F_G = -mg.
\end{equation} 

\item The buoyant force $F_B$ arises from the pressure of fluid acting on the floating object (see Fig.~\ref{buoyantforce}). With the outer unit normal of the cylinder $\hat{n}_c$ and the unit vertical upward pointing vector $\hat{k}$, the buoyant force has the form:
\begin{equation}
F_B = \hat{k} \cdot \int_{\vert\Sigma\vert} \vec{F}\,ds,
\end{equation}
where the centripetal component pressure $\vec{F} = \rho g y \hat{n}_c$ and $\vert\Sigma\vert$ denotes the wetted region.
  
\begin{figure}[h]
	\includegraphics[scale = 0.3]{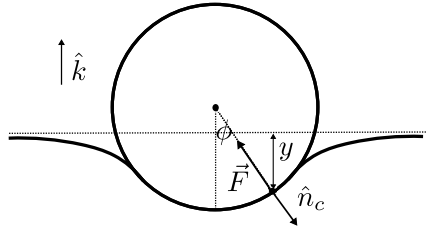}
	\caption{Computation of buoyant force.}
	\label{buoyantforce}
\end{figure}

$F_B$ can be calculated by integrating with respect to $\phi$ instead of $s$. As a result,\begin{equation}
F_B(\phi_0) = -4a\sqrt{\sigma \rho g}\cos\left(\frac{\phi_0+\gamma}{2}\right)\sin\phi_0-\frac{1}{2}\rho g a^2\sin2\phi_0+\rho g a^2\phi_0.
\end{equation}

With no surface tension, the divergence theorem leads to Archimedes' principle. But Archimedes' principle is not generally correct when the surface tension is present (see Appendix~\ref{Apx2}).

\item The surface tension force:

in 1805, Thomas Young\cite{thomasyoung} derived the formula to determine the contact angle $\gamma$ in terms of three surface tensions. Fig.~\ref{youngdia1} is known as ``Young's diagram''. Balancing the forces tangential to the solid gives
\begin{equation}
\cos \gamma = \frac{\sigma_1-\sigma_2}{\sigma},
\end{equation}
where $\sigma$ is the air/liquid surface tension, $\sigma_1$ and $\sigma_2$ are the air/solid and the liquid/solid surface tension, respectively. 

\begin{figure*}[h]
     \centering
     \subfloat[]{\includegraphics[scale = 0.3]{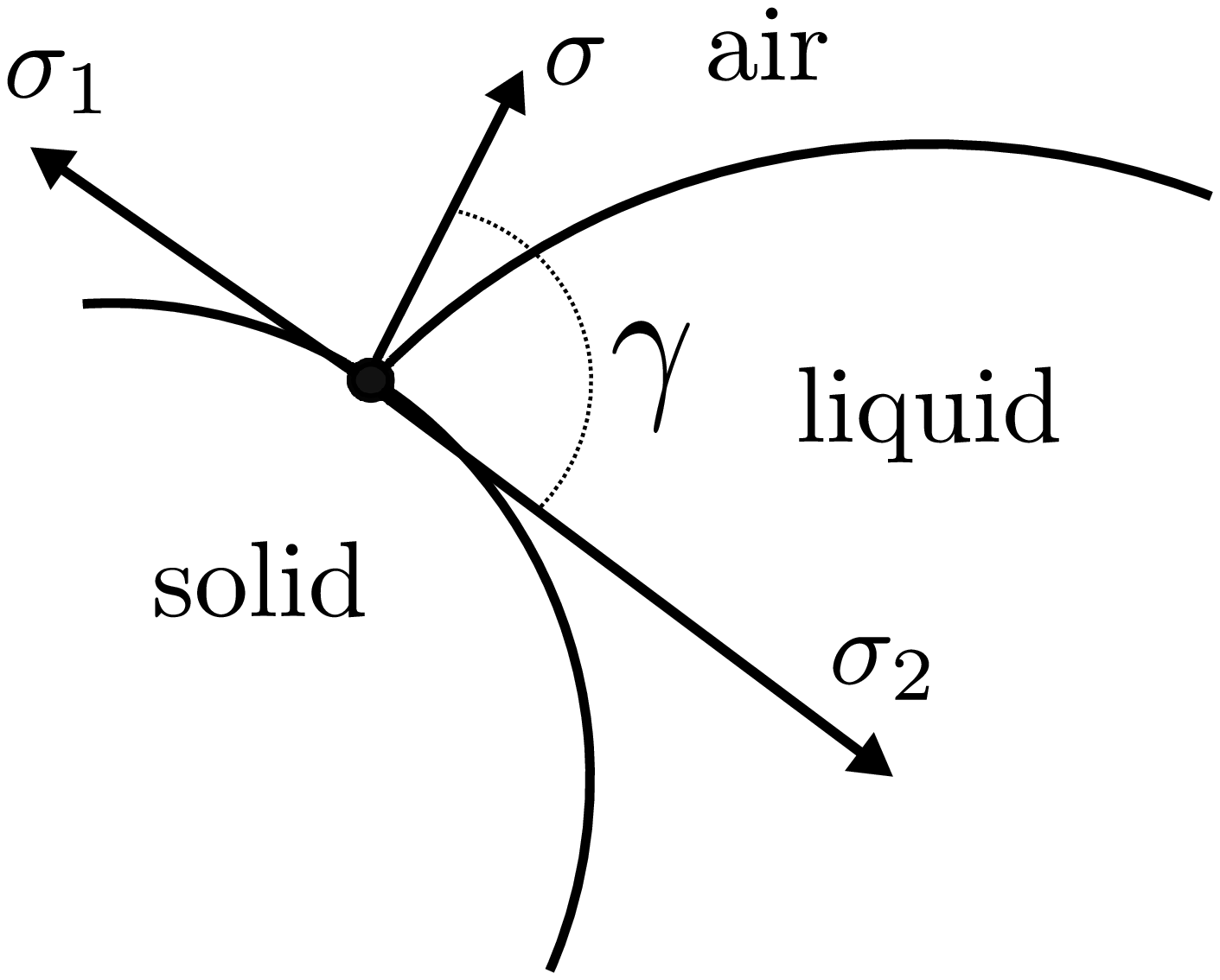}\label{youngdia1}}
     \hspace{2cm}
     \subfloat[]{\includegraphics[scale =0.3]{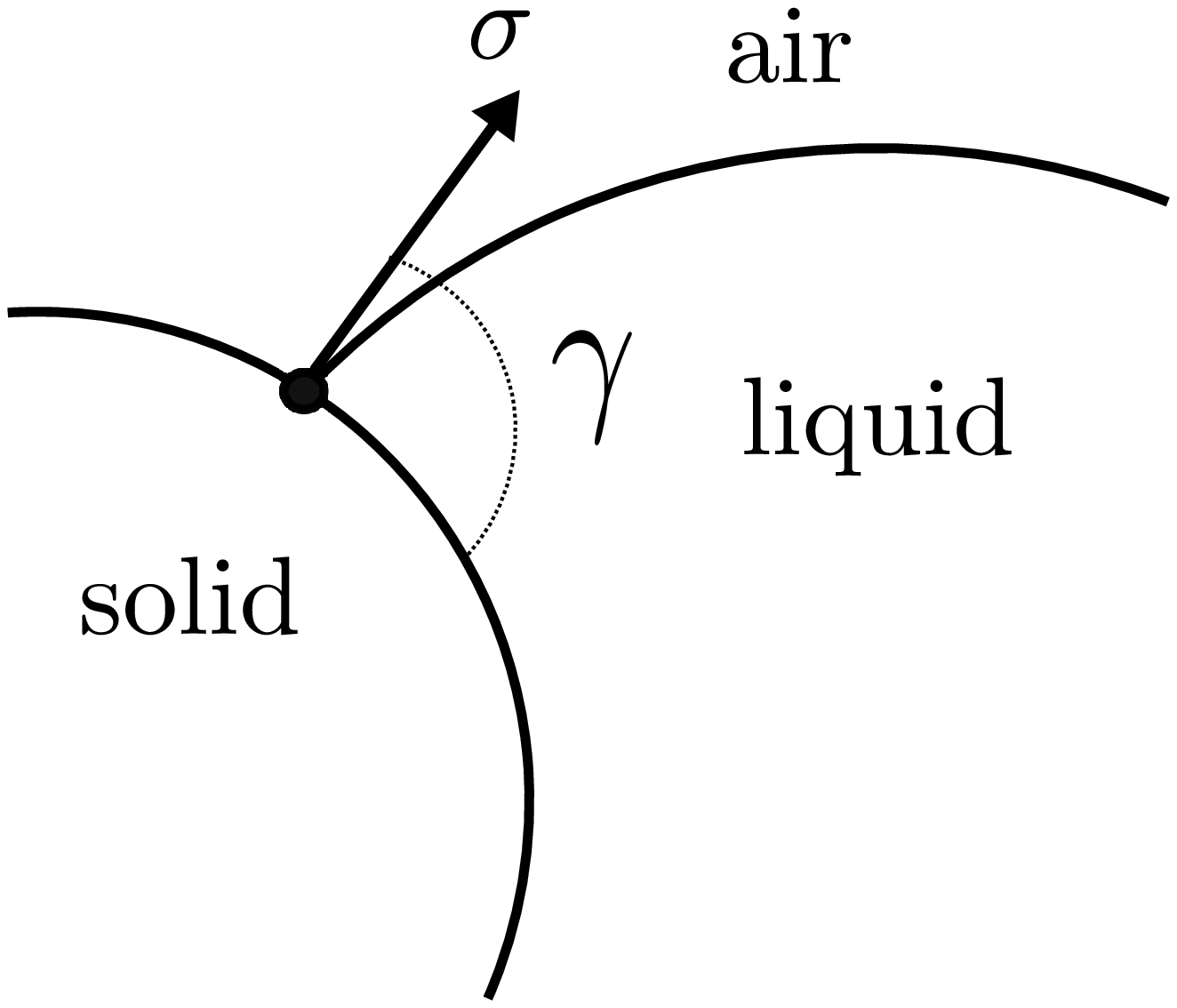}\label{youngdia2}}
     \caption{(a) Young's diagram and (b) its correction.}
\end{figure*}

The discussion of Young's diagram has gone on for centuries. Recently, Finn\cite{youngpara} gave a counterexample to show the incorrectness of Young's diagram. Instead of applying Young's diagram, we agree with Gifford and Scriven\cite{attraction}, Finn\cite{youngpara}, Bhatnagar and Finn\cite{MR2259294} that the surface tension acts only along the fluid interface. 

In our case, the vertical component of the surface tension is
\begin{equation}
F_{\sigma}(\phi_0) = -2\sigma\sin(\phi_0+\gamma).
\end{equation}

\end{enumerate}

Therefore, the full expression of $F_T(\phi_0)$ is 
\begin{eqnarray}
F_T(\phi_0) &=& F_G + F_\sigma + F_B \nonumber \\
&=& -mg-2\sigma \sin(\phi_0+\gamma)-4a\sqrt{\sigma \rho g}\cos\left(\frac{\phi_0+\gamma}{2}\right)\sin\phi_0 \nonumber \\
&& -\frac{1}{2}\rho g a^2\sin2\phi_0+\rho g a^2\phi_0.
\label{toforce}
\end{eqnarray}

\subsection{Relation between the Total Energy and the Total Force}\label{2drelation}

As minimizing the total energy $E_T(\phi_0)$ in Eq.~(\ref{totalenergy}) is laborious, we'll introduce a more convenient approach. Firstly, we observe the one-to-one correspondence between $h(\phi_0)$ and $\phi_0$, since $\frac{dh}{d\phi_0}<0$ on $\phi_0 \in \left[0,\pi\right]$ except $\phi_0 =\gamma = 0$ and $\phi_0 = \gamma = \pi$ cases (they are not physically realizable, see details in Sec~\ref{intersectioncase}).

One main result is the relation between $E_T$ and $F_T$, which follows by the chain rule: 
\begin{equation}
-\frac{dE_T}{d h} = -\frac{dE_T}{d \phi_0}\frac{d\phi_0}{d h} = F_T.
\label{cha2detdh}
\end{equation}

For details, see Appendix~\ref{Apxc}. The relation in (\ref{cha2detdh}) leads to the following equivalences. Since $\frac{d\phi_0}{d h} <0$ (except $\phi_0 =\gamma = 0$ and $\phi_0 = \gamma = \pi$), $\frac{dE_T}{d\phi_0}$ and $F_T$ have the same sign, that is
\begin{equation}
\sign\left(\frac{dE_T}{d\phi_0}\right)=\sign\left(F_T\right). 
\end{equation}

Assume that $\bar{\phi}_0 \in \left(0, \pi\right)$ is the critical point for $E_T\left(\phi_0\right)$, then
\begin{equation}
\frac{dE_T}{d\phi_0}\left(\bar{\phi}_0\right) = 0 \quad \Leftrightarrow \quad F_T\left(\bar{\phi}_0\right) = 0.
\label{equi}
\end{equation}
Thus the critical point $\bar{\phi}_0$ for $E_T\left(\phi_0\right)$ is equivalent to the force balance point $F_T\left(\bar{\phi}_0\right)=0$. We rearrange Eq.~(\ref{cha2detdh}) and differentiate with respect to $\phi_0$, then
\begin{equation}
-\frac{d^2 E_T}{d\phi_0^2} = \frac{dF_T}{d\phi_0}\frac{dh}{d\phi_0}+F_T\frac{d^2h}{d\phi_0^2}.
\label{2nddET}
\end{equation}

If we evaluate at $\bar{\phi}_0$, we have the following sign equivalence from Eq.~(\ref{2nddET}) 
\begin{equation}
\sign\left(\frac{d^2 E_T}{d \phi_0^2}\left(\bar{\phi_0}\right)\right)= \sign \left(\frac{dF_T}{d\phi_0}\left(\bar{\phi_0}\right)\right).
\end{equation}

Thus $\bar{\phi}_0$ is a local minimum if $\sign\left(\frac{d^2 E_T}{d \phi_0^2}\left(\bar{\phi_0}\right)\right) >0$ and $\bar{\phi}_0$ is a local maximum if $\sign\left(\frac{d^2 E_T}{d \phi_0^2}\left(\bar{\phi}_0\right)\right) <0$, equivalently,  
\begin{subequations}
\begin{equation}
\frac{dF_T}{d\phi_0}\left(\bar{\phi}_0\right) > 0 \quad \Rightarrow \quad \bar{\phi}_0 \text{ is locally stable}, \label{stableeq}
\end{equation}
\begin{equation}
\frac{dF_T}{d\phi_0}\left(\bar{\phi}_0\right) < 0 \quad \Rightarrow \quad \bar{\phi}_0 \text{ is locally unstable}. \label{unstableeq}
\end{equation}
\end{subequations}

With the equivalences above, we will focus on $F_T$ instead of $E_T$ in minimizing the total energy. The next stage is to find the force balance point. Two techniques, non-dimensionalization and Fourier decomposition, will be introduced. 

\begin{remark}\label{BFexample}
Bhatnagar and Finn\cite{MR2259294} gave the first example where a floating cylinder admits two equilibrium positions (we label the equilibria $\bar{\phi}_{01}<\bar{\phi}_{02}$). With parameters $\left\{g = 980\textrm{cm}/s^2, m = 1.2\textrm{g}, \rho = 1 \textrm{g}/\textrm{cm}^2, \sigma = 72 \textrm{dyn}/\textrm{cm}, \gamma = \frac{\pi}{2}, a = \frac{1}{\sqrt{\pi}} \textrm{cm}\right\}$, they assert $\bar{\phi}_{01}$ is unstable and $\bar{\phi}_{02}$ is stable. Here, we correct their stability assertion based on Eqs.~(\ref{stableeq}) and (\ref{unstableeq}). The smaller equilibrium point $\bar{\phi}_{01}$ is stable and the larger equilibrium point $\bar{\phi}_{02}$ is unstable. 
\end{remark}

\subsection{Two Independent Non-dimensional Parameters}
Bhatnagar and Finn\cite{MR2259294} introduced two dimensionless parameters:
\begin{equation}
\mathcal{A} = \frac{m}{a^2 \rho} \quad  \textrm{and} \quad \mathcal{B} = \frac{\rho g}{\sigma}a^2,
\end{equation}
where $\mathcal{A}$ also has the form $\mathcal{A}=\pi\frac{\rho_m}{\rho}$, where $\rho_m$ is the density difference between the cylinder and the air. $\mathcal{B}$ is known as the Bond number, which is the ratio of gravitational to surface tension forces. It'll be convenient to introduce $\mathcal{C} = \sqrt{\mathcal{B}}=\sqrt{\kappa} a$. The equation of the total force $F_T$ in (\ref{toforce}) can be expressed as
\begin{equation}
F_T(\phi_0) = \sigma\bigg[-\mathcal{A}\mathcal{C}^2-2\sin\left(\phi_0+\gamma\right)-4\mathcal{C}\cos\left(\frac{\phi_0+\gamma}{2}\right)\sin\phi_0-\frac{1}{2}\mathcal{C}^2\sin2\phi_0+\mathcal{C}^2\phi_0 \bigg].
\label{toforcem1}
\end{equation}

If we define a characteristic force as $F_c = 1\sigma$, where $1$ is a unit length of the horizontal cylinder, we have the dimensionless form of the total force  $\hat{F}_T$:
\begin{equation}
\hat{F}_T(\phi_0) = -\mathcal{A}\mathcal{C}^2-2\sin\left(\phi_0+\gamma\right)-4\mathcal{C}\cos\left(\frac{\phi_0+\gamma}{2}\right)\sin\phi_0-\frac{1}{2}\mathcal{C}^2\sin2\phi_0+\mathcal{C}^2\phi_0.
\label{toforcem2}
\end{equation}

\subsection{Trigonometric Series}

We write $\hat{F}_T(\phi_0)=\bar{F}_T(\phi_0)-\mathcal{A}\mathcal{C}^2+\mathcal{C}^2\phi_0$. The total force $\hat{F}_T$ in (\ref{toforcem2}) is mainly comprised of trigonometric functions sine and cosine. 
A Fourier decomposition can be applied and $\bar{F}_T$ can be written as the trigonometric series in terms of 
\begin{equation}
\left\{\sin\frac{\phi_0}{2}, \cos\frac{\phi_0}{2}, \sin\phi_0, \cos\phi_0,\sin\frac{3\phi_0}{2}, \cos\frac{3\phi_0}{2}, \sin2\phi_0, \cos2\phi_0\right\}.
\end{equation}

The projection formulas give the expression of the coefficients. 
\begin{subequations}
\begin{equation}
a_n = \frac{1}{2\pi}\int_0^{4\pi} \bar{F}_T(\phi_0)\cos\left(\frac{n\phi_0}{2}\right) \,d\phi_0,\quad  n\in \{1,2,3,4\},
\end{equation}
\begin{equation}
b_n = \frac{1}{2\pi}\int_0^{4\pi} \bar{F}_T(\phi_0)\sin\left(\frac{n\phi_0}{2}\right) \,d\phi_0,\quad n\in \{1,2,3,4\},
\end{equation}
\end{subequations}
where $a_n$ is the coefficient of $\cos\left(\frac{n\phi_0}{2}\right)$ and $b_n$ is the coefficient of $\sin\left(\frac{n\phi_0}{2}\right)$.

The total force equation $\hat{F}_T$ in (\ref{toforcem2}) can be transformed to the following
\begin{eqnarray}
\hat{F}_T(\phi_0) &=& -{\mathcal{A}}{\mathcal{C}}^2-2\mathcal{C}\cos\frac{\gamma}{2}\sin\frac{\phi_0}{2}+2\mathcal{C}\sin\frac{\gamma}{2}\cos\frac{\phi_0}{2}- 2\cos\gamma\sin\phi_0 \nonumber \\
&&-2\sin\gamma\cos\phi_0-2\mathcal{C}\cos\frac{\gamma}{2}\sin\frac{3\phi_0}{2}-2\mathcal{C}\sin\frac{\gamma}{2}\cos\frac{3\phi_0}{2} \nonumber \\
&&-\frac{1}{2}{\mathcal{C}}^2\sin2\phi_0+{\mathcal{C}}^2\phi_0,
\end{eqnarray}
where $\mathcal{A}$ only appears in constant term, thus $\frac{d\hat{F}_T}{d\phi_0}$ does not depend on $\mathcal{A}$. We write $\frac{d\hat{F}_T}{d\phi_0}(\phi_0;\mathcal{C})$. 

\section{Stability Behavior}\label{stabe}

We wish to study the stability behavior of our floating cylinder system. First of all, we have to find the equilibria based on the equivalence relation in Eq.~(\ref{equi}), that is the force balance point $\bar{\phi_0}$. To analyze $\hat{F}_T$, we consider  dimensionless parameters $\mathcal{A}>0$ and $\mathcal{C}>0$ to have physical meaning, the contact angle $\gamma \in \left[0, \pi\right]$ and the wetting angle $\phi_0 \in \left[0,\pi\right]$. The discussion will be divided into three cases: $\gamma = \frac{\pi}{2}$, $\gamma > \frac{\pi}{2}$ and $\gamma < \frac{\pi}{2}$. In this section, the inequalities in Eqs.~(\ref{matlabcheck2}),(\ref{matlabcheck3}),(\ref{matlabcheck4}),(\ref{matlabcheck5}),(\ref{matlabcheck6}),(\ref{matlabcheck7}),(\ref{matlabcheck8}) and (\ref{matlabcheck9}) are checked with Matlab.


\subsection{The Case $\gamma = \frac{\pi}{2}$}

\begin{theorem}\label{behaeqpiby2}
When $\gamma = \frac{\pi}{2}$, $\hat{F}_T$ curve has the following two properties:
\begin{enumerate}
\item $\hat{F}_T$ is centrally symmetric with respect to the point $\left(\frac{\pi}{2},\hat{F}_T\left(\frac{\pi}{2}\right)\right)$.
\item There are two critical points for $\hat{F}_T(\phi_0)$, one lies in $(0, \frac{\pi}{2})$ and the other is in $(\frac{\pi}{2}, \pi)$.
\end{enumerate}
\label{behapipy2}
\end{theorem}

\begin{proof}
\begin{enumerate}
\item Choose $\phi_0 \in \left[0, \pi\right]$ and so $\pi-\phi_0 \in \left[0, \pi\right]$. Moreover, 
\begin{equation*}
\hat{F}_T(\phi_0)+\hat{F}_T(\pi-\phi_0) = -2\mathcal{A}\mathcal{C}^2+\mathcal{C}^2\pi
= 2 \hat{F}_T\left(\frac{\pi}{2}\right).
\end{equation*}

\item We'll apply the intermediate value theorem to $\frac{d \hat{F}_T}{d \phi_0}$,
\begin{equation*}
\frac{d \hat{F}_T}{d \phi_0} = 2\sin\phi_0+2\mathcal{C}\sin\left(\frac{\phi_0}{2}+\frac{\pi}{4}\right)\sin\phi_0-4\mathcal{C}\cos\left(\frac{\phi_0}{2}+\frac{\pi}{4}\right)\cos\phi_0-\mathcal{C}^2\cos(2\phi_0)+\mathcal{C}^2.
\end{equation*}

We have a sign change of $\frac{d\hat{F}_T}{d \phi_0}$ on both subintervals $\left[0,\frac{\pi}{2}\right]$ and $\left[\frac{\pi}{2},\pi\right]$, since   
\begin{equation*}
\frac{d \hat{F}_T}{d \phi_0}(0) = -2\sqrt{2}\mathcal{C}<0,\quad \frac{d \hat{F}_T}{d \phi_0}\left(\frac{\pi}{2}\right) = 2(1+\mathcal{C}+\mathcal{C}^2)>0 \quad \textrm{and} \quad  \frac{d \hat{F}_T}{d \phi_0}(\pi) = -2\mathcal{C}<0. 
\end{equation*}

Moreover, $\frac{d \hat{F}_T}{d \phi_0}$ is strictly increasing on $\left(0, \frac{\pi}{2}\right)$, which follows from
\begin{equation}\label{matlabcheck1}
\frac{d^2 \hat{F}_T}{d{\phi^2_0}} = 2\cos\phi_0+\mathcal{C}\cos\left(\frac{\phi_0}{2}+\frac{\pi}{4}\right)\sin\phi_0+4\mathcal{C}\sin\left(\frac{3\phi_0}{2}+\frac{\pi}{4}\right)+2\mathcal{C}^2\sin(2\phi_0)>0.
\end{equation}

Since $\frac{d\hat{F}_T}{d\phi_0}$ is continuous on $\left[0,\frac{\pi}{2}\right]$, $\frac{d \hat{F}_T}{d \phi_0}(0)<0$ and $\frac{d \hat{F}_T}{d \phi_0}\left(\frac{\pi}{2}\right)>0$, $\hat{F}_T$ admits exactly one critical point in $\left(0, \frac{\pi}{2}\right)$ based on the intermediate value theorem. By centrally symmetry, $\hat{F}_T$ also admits another critical point in $\left(\frac{\pi}{2}, \pi\right)$.
\end{enumerate}
\end{proof}

\begin{figure}[h]
\includegraphics[scale = 0.4]{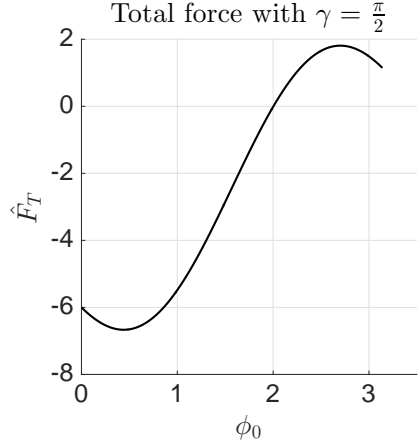}
\caption{$\hat{F}_T$ curve with parameters $\mathcal{A}=4$ and $\mathcal{C}=1$.}
\label{tf1}
\end{figure}

\begin{figure}[h]
\includegraphics[scale = 0.4]{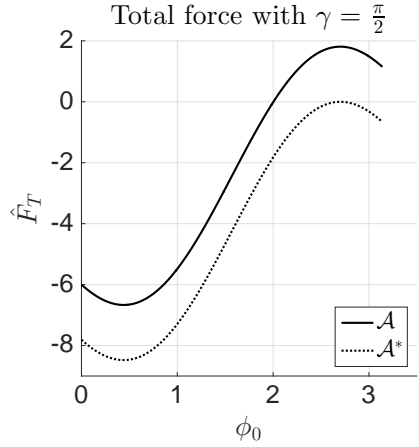}
\caption{$\hat{F}_T$ curves with $\mathcal{A} = 4$, $\mathcal{C}=1$ and $\mathcal{A}^*=5.0893$, $\mathcal{C}=1$.}
\label{tf2}
\end{figure}

In addition, the number of equilibria and their stability can be determined, as follows.

\begin{theorem}\label{eqstablepiby2}
For $\gamma = \frac{\pi}{2}$, $\hat{F}_T$ admits at most two equilibrium points (we label the equilibria $\bar{\phi}_{01}<\bar{\phi}_{02}$), the smaller $\bar{\phi}_{01}$ is stable and the larger $\bar{\phi}_{02}$ is unstable, where $\bar{\phi}_{02} \in (\frac{\pi}{2},\pi)$ if it exists. If $\hat{F}_T$ admits only one equilibrium point (we label it $\bar{\phi}_0$), $\bar{\phi}_0$ is stable if $\frac{d\hat{F}_T}{d\phi_0}(\bar{\phi}_0)>0$ and it is unstable if $\frac{d\hat{F}_T}{d\phi_0}(\bar{\phi}_0)=0$.
\end{theorem}

\begin{proof}
Based on Theorem \ref{behapipy2}, $\hat{F}_T$ decreases at the beginning then reaches the first critical point, and then $\hat{F}_T$ increases until reaching the second critical point, finally $\hat{F}_T$ decreases (see Fig.~\ref{tf1}). 

Moreover, at $\phi_0 = 0$, we have
\begin{equation}\label{f0value}
\hat{F}_T(0)=-\mathcal{A}\mathcal{C}^2-2\sin\gamma<0.
\end{equation}
At $\phi_0 = \pi$, we have
\begin{equation}\label{fpivalue}
\hat{F}_T(\pi) = 2\sin\gamma + \mathcal{C}^2(\pi-\mathcal{A}).
\end{equation}

From Eqs.~(\ref{f0value}) and (\ref{fpivalue}), we obtain 
\begin{equation}\label{endpts}
\hat{F}_T(\pi)> \hat{F}_T(0) \quad \textrm{for} \quad \textrm{arbitrary } \gamma. 
\end{equation}

The behavior of $\hat{F}_T$ shows it admits at most two equilibrium points. If $\bar{\phi}_{02}$ exists, it would be always greater than the second critical point of $\hat{F}_T(\phi_0)$, thus $\bar{\phi}_{02} \in (\frac{\pi}{2},\pi)$. Their stability immediately comes from the criteria (\ref{stableeq}) and (\ref{unstableeq}). For the $\frac{d\hat{F}_T}{d\phi_0}(\bar{\phi}_0)=0$ case, $\bar{\phi}_0$ is unstable, since $\frac{d\hat{F}_T}{d\phi_0}(\bar{\phi}_0-\epsilon)>0$ and $\frac{d\hat{F}_T}{d\phi_0}(\bar{\phi}_0+\epsilon)<0$, for small $\epsilon>0$.
\end{proof}

We also consider how the values of $\mathcal{A}$ affect the number of equilibria of $\hat{F}_T$. Since $\mathcal{A}$ only appears in the constant term of $\hat{F}_T$, if the value of $\mathcal{A}$ increases, the curve of $\hat{F}_T$ will shift down (see Fig.~\ref{tf2}). Given the value of $\mathcal{C}$, we define $\mathcal{A}^*$ such that
\begin{equation}
\hat{F}_T(\phi_0^*;\mathcal{A}^*)=0,
\label{asymeqn}
\end{equation}
where $\phi_0^*>\frac{\pi}{2}$ is the second critical point of $\hat{F}_T$. Unfortunately, $\mathcal{A}^*$ has to be found numerically. The following table~\ref{eqstabletb} shows the number of equilibria and their stability for different values of $\mathcal{A}$. In addition, the number of equilibria can also be shown in $\mathcal{C}$ vs $\mathcal{A}$ Figures. The details will be discussed in Sec.~\ref{AC}.

\begin{table}[h]
\caption{Number of equilibria and their stability for different $\mathcal{A}$ for $\gamma = \frac{\pi}{2}$.}
\begin{ruledtabular}
\begin{tabular}{ccc}
Range of $\mathcal{A}$&Number of Equilibria&Stability\\
\hline
$0<\mathcal{A}<\frac{2}{\mathcal{C}^2}+\pi$ & 1 & stable\\
$\frac{2}{\mathcal{C}^2}+\pi\leq \mathcal{A}<\mathcal{A}^*$ & 2 & $\bar{\phi}_{01}$ is stable, $\bar{\phi}_{02}$ is unstable \\
$\mathcal{A}=\mathcal{A}^*$ & 1 & unstable\footnote{Since $\frac{d\hat{F}_T}{d\phi_0}(\bar{\phi}_0-\epsilon)>0$ and $\frac{d\hat{F}_T}{d\phi_0}(\bar{\phi}_0+\epsilon)<0$, for small $\epsilon>0$.} \\
$\mathcal{A}>\mathcal{A}^*$ & 0 & NA
\end{tabular}
\end{ruledtabular}
\label{eqstabletb}
\end{table}

\subsection{The Case $\gamma > \frac{\pi}{2}$}

When $\gamma > \frac{\pi}{2}$, the behavior of $\hat{F}_T(\phi_0)$ depends on $\sign\left(\frac{d\hat{F}_T}{d\phi_0}\right)$ at the end point $\phi_0=0$. This leads to the following theorem: 
\begin{theorem}\label{behagthanpiby2}
For $\gamma > \frac{\pi}{2}$, there are two types of behavior of the total force $\hat{F}_T$ curve. 
\begin{enumerate}
\item If $\frac{d \hat{F}_T}{d \phi_0}(0)<0$, there are two critical points, one lies in $\left(0, \frac{\pi}{4}\right)$, the other lies in $\left(\frac{\pi}{2}, \pi\right)$. $\hat{F}_T$ decreases to the first critical point, then increases to the second critical point, and then decreases. 
\item If $\frac{d \hat{F}_T}{d \phi_0}(0) \geq 0$, there is only one critical point in $\left(\frac{\pi}{2}, \pi\right)$. $\hat{F}_T$ increases to the only critical point and then decreases.
\end{enumerate}
\end{theorem}

\begin{proof}
\begin{enumerate}
\item We firstly consider the above two cases for $\phi_0 \in \left[0, \frac{\pi}{4}\right]$. If $\frac{d \hat{F}_T}{d \phi_0}(0)<0$, we have 
\begin{equation}
2\mathcal{C}\cos\frac{\gamma}{2}+\cos\gamma >0 \quad \Leftrightarrow \quad  2\cos\gamma\sin\phi_0 > -4\mathcal{C}\cos\frac{\gamma}{2}\sin\phi_0,
\label{piby2ineq}
\end{equation}
where $\gamma \neq \pi$ and $\phi_0 \neq 0$. The inequality in (\ref{piby2ineq}) gives the underlined terms in the following 
\begin{eqnarray}\label{matlabcheck2}
\frac{d^2 \hat{F}_T}{d{\phi^2_0}} &=& \frac{\mathcal{C}}{2}\left(\cos\frac{\gamma}{2}\sin\frac{\phi_0}{2}-\sin\frac{\gamma}{2}\cos\frac{\phi_0}{2}+9\cos\frac{\gamma}{2}\sin\frac{3\phi_0}{2}+9\sin\frac{\gamma}{2}\cos\frac{3\phi_0}{2}\right)\nonumber\\
&&+2\mathcal{C}^2\sin2\phi_0+\underline{2\cos\gamma\sin\phi_0}+2\sin\gamma\cos\phi_0\nonumber\\
&>&\frac{\mathcal{C}}{2}\left(\cos\frac{\gamma}{2}\sin\frac{\phi_0}{2}-\sin\frac{\gamma}{2}\cos\frac{\phi_0}{2}-\underline{8\cos\frac{\gamma}{2}\sin\phi_0}+9\cos\frac{\gamma}{2}\sin\frac{3\phi_0}{2}+9\sin\frac{\gamma}{2}\cos\frac{3\phi_0}{2}\right)\nonumber\\
&&+2\mathcal{C}^2\sin2\phi_0+4\sin\gamma\cos\phi_0>0.
\end{eqnarray}

Inequality~(\ref{matlabcheck2}) is obtained by showing that the coefficients of different powers of $\mathcal{C}$ are positive. Matlab is used to check that the coefficient of $\mathcal{C}$ is positive. Moreover, $\frac{d^2 \hat{F}_T}{d{\phi^2_0}}(0) = 4\mathcal{C}\sin\frac{\gamma}{2}+2\sin\gamma >0$. Therefore, $\frac{d\hat{F}_T}{d\phi_0}$ is increasing on $\left[0, \frac{\pi}{4}\right]$. 

Applying the intermediate value theorem again, with $\frac{d\hat{F}_T}{d\phi_0}(0)<0$ and 
\begin{eqnarray}\label{matlabcheck3}
\frac{d\hat{F}_T}{d\phi_0}(\frac{\pi}{4}) &=&\mathcal{C}\left(-\cos\frac{\gamma}{2}\cos\frac{\pi}{8}-\sin\frac{\gamma}{2}\sin\frac{\pi}{8}-3\cos\frac{\gamma}{2}\cos\frac{3\pi}{8}+3\sin\frac{\gamma}{2}\sin\frac{3\pi}{8}\right)\nonumber\\
&&+\mathcal{C}^2+\sqrt{2}\left(\sin\gamma-\cos\gamma\right)>0.
\end{eqnarray}
Thus, we conclude $\hat{F}_T$ has a critical point, which lies in $\left(0,\frac{\pi}{4}\right)$.

If $\frac{d \hat{F}_T}{d \phi_0}(0)\geq0$, we have 
\begin{equation}
2\mathcal{C}\cos\frac{\gamma}{2}+\cos\gamma\leq 0,
\label{thmcon}
\end{equation}
where the condition in (\ref{thmcon}) implies the monotonicity of $\hat{F}_T$ in $\left[0,\frac{\pi}{4}\right]$:
\begin{eqnarray}\label{matlabcheck4}
\frac{d\hat{F}_T}{d\phi_0} &=& \mathcal{C}^2-\mathcal{C}^2\cos2\phi_0
-2\cos\gamma\cos\phi_0+2\sin\gamma\sin\phi_0\nonumber\\
&&+\mathcal{C}\left(-\cos\frac{\gamma}{2}\cos\frac{\phi_0}{2}-\sin\frac{\gamma}{2}\sin\frac{\phi_0}{2}-3\cos\frac{\gamma}{2}\cos\frac{3\phi_0}{2}+3\sin\frac{\gamma}{2}\sin\frac{3\phi_0}{2}\right)\nonumber\\
&\geq& \left(2\sin\gamma\sin\phi_0-2\cos\gamma\cos\phi_0+\frac{1}{2}\cos\gamma\cos\frac{\phi_0}{2}+\frac{3}{2}\cos\gamma\cos\frac{3\phi_0}{2}\right) \nonumber \\
&&+\mathcal{C}^2(1-\cos2\phi_0)+\mathcal{C}\sin\frac{\gamma}{2}\left(3\sin\frac{3\phi_0}{2}-\sin\frac{\phi_0}{2}\right)\geq0,
\end{eqnarray}

the equality only holds for $\phi_0=0$. Therefore, $\hat{F}_T$ increases on $\left[0, \frac{\pi}{4}\right]$.

\item Next we consider $\phi_0 \in \left[\frac{\pi}{4}, \frac{\pi}{2}\right]$, 
\begin{eqnarray}\label{matlabcheck5}
\frac{d \hat{F}_T}{d \phi_0} &=& \mathcal{C}\left(-\cos\frac{\gamma}{2}\cos\frac{\phi_0}{2}-\sin\frac{\gamma}{2}\sin\frac{\phi_0}{2}-3\cos\frac{\gamma}{2}\cos\frac{3\phi_0}{2}+3\sin\frac{\gamma}{2}\sin\frac{3\phi_0}{2}\right) \nonumber\\
&& +(2\sin\gamma\sin\phi_0-2\cos\gamma\cos\phi_0) +\mathcal{C}^2(1-\cos 2\phi_0) >0.
\end{eqnarray}

Hence, $\hat{F}_T$ is increasing on $\left[\frac{\pi}{4}, \frac{\pi}{2}\right]$.

\item Finally, considering $\phi_0 \in \left[\frac{\pi}{2}, \pi\right]$, we have
\begin{eqnarray}\label{matlabcheck6}
\frac{d^2 \hat{F}_T}{d{\phi^2_0}} &=& \frac{\mathcal{C}}{2}\left(\cos\frac{\gamma}{2}\sin\frac{\phi_0}{2}-\sin\frac{\gamma}{2}\cos\frac{\phi_0}{2}+9\cos\frac{\gamma}{2}\sin\frac{3\phi_0}{2}+9\sin\frac{\gamma}{2}\cos\frac{3\phi_0}{2}\right)\nonumber\\
&&+2\mathcal{C}^2\sin2\phi_0+2(\cos\gamma\sin\phi_0+\sin\gamma\cos\phi_0)<0.
\end{eqnarray}

Moreover, at $\phi_0 = \frac{\pi}{2}$ and $\phi_0 = \pi$,
\begin{eqnarray*}
\frac{d\hat{F}_T}{d\phi_0}(\pi) &=& -4\mathcal{C}\sin\frac{\gamma}{2}+2\cos\gamma<0,\\
\frac{d \hat{F}_T}{d \phi_0}\left(\frac{\pi}{2}\right) &=& \sqrt{2}\mathcal{C}\cos\frac{\gamma}{2}+\sqrt{2}\mathcal{C}\sin\frac{\gamma}{2}+2\sin\gamma+2\mathcal{C}^2 >0.
\end{eqnarray*}

Thus $\frac{d \hat{F}_T}{d \phi_0}$ is monotone decreasing on $\left[\frac{\pi}{2}, \pi\right]$. By the intermediate value theorem, there exists a $\phi^* \in \left(\frac{\pi}{2}, \pi\right)$ such that $\frac{d\hat{F}_T}{d\phi_0}(\phi^*) = 0$. Therefore, $\hat{F}_T(\phi_0)$ increases on $[0,\phi^*]$ and then decreases on $[\phi^*,\pi]$.
\end{enumerate}
\end{proof}

\begin{figure}[h!]
\includegraphics[scale = 0.4]{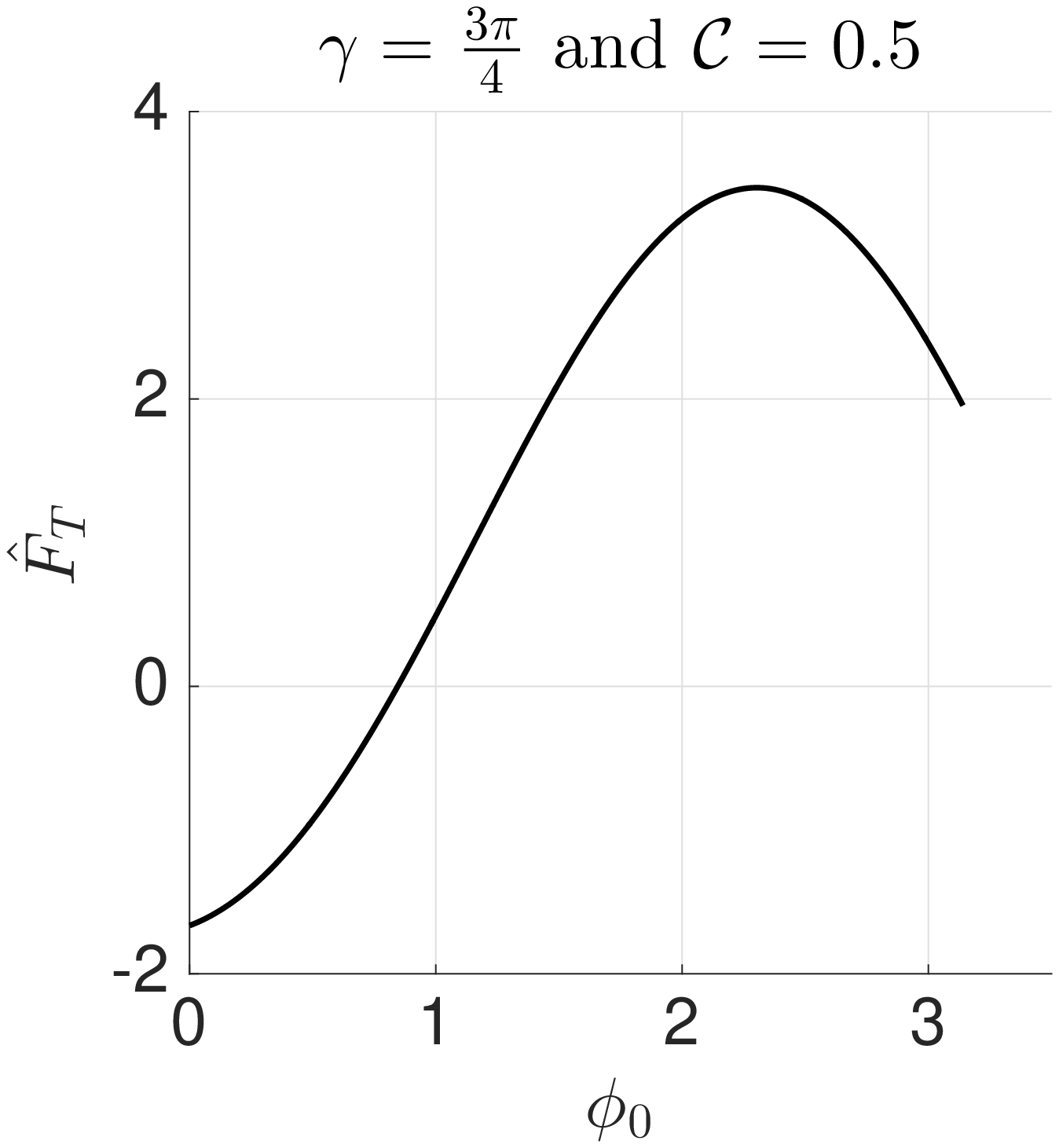}
\caption{$\hat{F}_T$ curve with parameters $\gamma = \frac{3\pi}{4}$, $\mathcal{C}=0.5$ and $\mathcal{A}=1$ such that $\frac{d\hat{F}_T}{d\phi_0}(0)\geq 0$. }
\label{case1}
\end{figure}

\begin{figure}[h!]
\includegraphics[scale = 0.4]{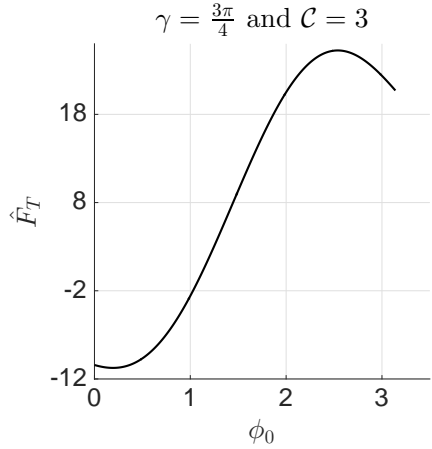}
\caption{$\hat{F}_T$ curve with parameters $\gamma = \frac{3\pi}{4}$, $\mathcal{C}=3$ and $\mathcal{A}=1$ such that $\frac{d\hat{F}_T}{d\phi_0}(0)<0$. }
\label{case2}
\end{figure}

Based on Theorem \ref{behagthanpiby2}, there are two types of behavior of $\hat{F}_T$ for $\gamma >\frac{\pi}{2}$: two typical examples of those cases are shown in Fig.~\ref{case1} and Fig.~\ref{case2}. The following Theorem shows the number of equilibria and their stability for $\gamma >\frac{\pi}{2}$.

\begin{theorem}\label{eqstablegtpiby2}
For $\gamma >\frac{\pi}{2}$, $\hat{F}_T$ admits at most two equilibrium points, the smaller $\bar{\phi}_{01}$ is stable and the larger $\bar{\phi}_{02}$ is unstable, where $\bar{\phi}_{02} \in (\frac{\pi}{2},\pi]$ if it exists. If $\hat{F}_T$ admits only one equilibrium point, $\bar{\phi}_0$ is stable if $\frac{d\hat{F}_T}{d\phi_0}(\bar{\phi}_0)>0$ and it is unstable if $\frac{d\hat{F}_T}{d\phi_0}(\bar{\phi}_0)=0$.
\end{theorem}
\begin{proof}
It is similar to Theorem~\ref{eqstablepiby2}. The result comes from Eq.~(\ref{endpts}), Theorem \ref{behagthanpiby2} and the criteria (\ref{stableeq}) and (\ref{unstableeq}).
\end{proof}

\subsection{The Case $\gamma < \frac{\pi}{2}$}

When $\gamma < \frac{\pi}{2}$, the behavior of $\hat{F}_T(\phi_0)$ depends on $\sign\left(\frac{d\hat{F}_T}{d\phi_0}\right)$ at end point $\phi_0 = \pi$. We obtain the following theorem: 

\begin{theorem}\label{behalthanpiby2}
For $\gamma < \frac{\pi}{2}$, there are two types of behavior of the total force $\hat{F}_T$ curve. 
\begin{enumerate}
\item If $\frac{d \hat{F}_T}{d \phi_0}(\pi)<0$, there are two critical points, one lies in $\left(0,\frac{\pi}{2}\right)$, the other lies in $\left(\frac{3\pi}{4}, \pi\right)$. $\hat{F}_T$ decreases to the first critical point, then increases to the second critical point, and then decreases. 
\item If $\frac{d \hat{F}_T}{d \phi_0}(\pi) \geq 0$, there is only one critical point in $\left(0,\frac{\pi}{2}\right)$. $\hat{F}_T$ decreases to the only critical point and then increases.
\end{enumerate}
\end{theorem}

\begin{proof}
\begin{enumerate}
\item We first consider $\phi_0 \in \left[0, \frac{\pi}{2}\right]$, with 
\begin{eqnarray*}
\frac{d\hat{F}_T}{d\phi_0}(0) &=& -4\mathcal{C}\cos\frac{\gamma}{2}-2\cos\gamma  < 0,\\
\frac{d \hat{F}_T}{d \phi_0}\left(\frac{\pi}{2}\right) &=& \sqrt{2}\mathcal{C}\cos\frac{\gamma}{2}+\sqrt{2}\mathcal{C}\sin\frac{\gamma}{2}+2\sin\gamma+2\mathcal{C}^2  >0.
\end{eqnarray*}
Moreover, 
\begin{eqnarray}\label{matlabcheck7}
\frac{d^2 \hat{F}_T}{d{\phi^2_0}} &=& \frac{\mathcal{C}}{2}\left(\cos\frac{\gamma}{2}\sin\frac{\phi_0}{2}-\sin\frac{\gamma}{2}\cos\frac{\phi_0}{2}+9\cos\frac{\gamma}{2}\sin\frac{3\phi_0}{2}+9\sin\frac{\gamma}{2}\cos\frac{3\phi_0}{2}\right)\nonumber \\
&&+2\mathcal{C}^2\sin2\phi_0+2(\cos\gamma\sin\phi_0+\sin\gamma\cos\phi_0)\geq0,
\end{eqnarray}

and equality holds only when both $\phi_0$ and $\gamma$ are $0$. Therefore, $\frac{d\hat{F}_T}{d\phi_0}$ increases on $\left[0, \frac{\pi}{2}\right]$. By the intermediate value theorem, there exists $\phi_0^* \in \left(0, \frac{\pi}{2}\right)$ such that $\frac{d\hat{F}_T}{d\phi_0}\left(\phi_0^*\right)=0$. As a result, $\hat{F}_T\left(\phi_0\right)$ decreases then reaches the critical point, and then increases.
\item Next, we consider $\phi_0 \in \left[\frac{\pi}{2}, \frac{3\pi}{4}\right]$, 
\begin{eqnarray}\label{matlabcheck8}
\frac{d\hat{F}_T}{d\phi_0} &=&\mathcal{C}\left(-\cos\frac{\gamma}{2}\cos\frac{\phi_0}{2}-\sin\frac{\gamma}{2}\sin\frac{\phi_0}{2}-3\cos\frac{\gamma}{2}\cos\frac{3\phi_0}{2}+3\sin\frac{\gamma}{2}\sin\frac{3\phi_0}{2}\right)\nonumber\\
&&+\mathcal{C}^2(1-\cos2\phi_0)+2(\sin\gamma\sin\phi_0-\cos\gamma\cos\phi_0)> 0.
\end{eqnarray}
Therefore, $\hat{F}_T(\phi_0)$ increases on $\left[\frac{\pi}{2}, \frac{3\pi}{4}\right]$.
\item Finally, we consider $\phi_0 \in \left[\frac{3\pi}{4}, \pi\right]$. We distinguish the following cases:
\begin{enumerate}
\item If $\frac{d\hat{F}_T}{d\phi_0}(\pi)\geq 0$, we have 
\begin{equation}
\cos\gamma \geq 2\mathcal{C}\sin\frac{\gamma}{2} \quad \Leftrightarrow \quad -\mathcal{C} \geq -\frac{\cos\gamma}{2\sin\frac{\gamma}{2}} \quad \textrm{with} \quad \gamma \neq 0.
\label{thmcon2}
\end{equation}

The inequality in (\ref{thmcon2}) leads to the following result 
\begin{eqnarray}\label{matlabcheck9}
\frac{d\hat{F}_T}{d\phi_0} &=& \mathcal{C}^2-\mathcal{C}^2\cos2\phi_0
-2\cos\gamma\cos\phi_0+2\sin\gamma\sin\phi_0\nonumber\\
&&-\mathcal{C}\left(\cos\frac{\gamma}{2}\cos\frac{\phi_0}{2}+\sin\frac{\gamma}{2}\sin\frac{\phi_0}{2}+3\cos\frac{\gamma}{2}\cos\frac{3\phi_0}{2}-3\sin\frac{\gamma}{2}\sin\frac{3\phi_0}{2}\right)\nonumber\\
&\geq& \mathcal{C}^2(1-\cos2\phi_0)+2(\sin\gamma\sin\phi_0-\cos\gamma\cos\phi_0)\nonumber\\
&&-\frac{\cos\gamma}{2\sin\frac{\gamma}{2}}\left(\cos\frac{\gamma}{2}\cos\frac{\phi_0}{2}+\sin\frac{\gamma}{2}\sin\frac{\phi_0}{2}+3\cos\frac{\gamma}{2}\cos\frac{3\phi_0}{2}-3\sin\frac{\gamma}{2}\sin\frac{3\phi_0}{2}\right)\nonumber\\
&>& 0,
\end{eqnarray}

In addition, if $\gamma = 0$, $\frac{d\hat{F}_T}{d\phi_0}>0$ as well. Hence $\hat{F}_T$ increases on $\left[\frac{3\pi}{4},\pi\right]$.

\item If $\frac{d\hat{F}_T}{d\phi_0}(\pi)<0$, then $\frac{d^2 \hat{F}_T}{d{\phi^2_0}}<0$ on $\left[\frac{3\pi}{4}, \pi \right]$. At the other end point $\frac{d\hat{F}_T}{d\phi_0}\left(\frac{3\pi}{4}\right)>0$. By the intermediate value theorem, $\hat{F}_T$ admits one critical point $\phi_0^* \in \left(\frac{3\pi}{4}, \pi\right)$. Therefore $\hat{F}_T(\phi_0)$ increases and reaches the critical point, then decreases at $\left[\frac{3\pi}{4}, \pi\right]$. 
\end{enumerate}

\end{enumerate}
\end{proof}

\begin{figure}[h!]
\includegraphics[scale = 0.4]{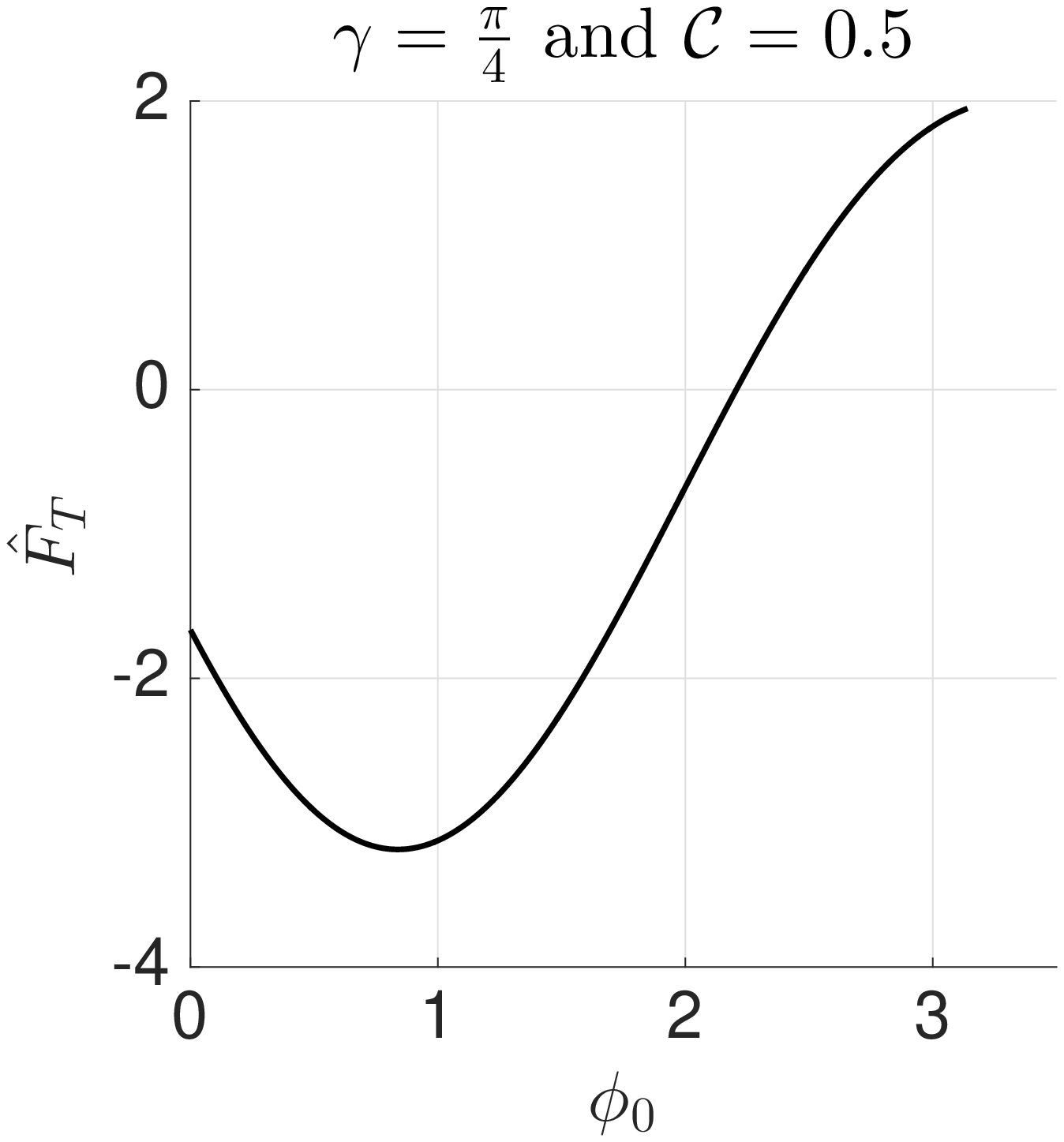}
\caption{$\hat{F}_T$ curve with parameters $\gamma = \frac{\pi}{4}$, $\mathcal{C}=0.5$ and $\mathcal{A}=1$ such that $\frac{d\hat{F}_T}{d\phi_0}(\pi)\geq 0$. }
\label{case3}
\end{figure}

\begin{figure}[h!]
\includegraphics[scale = 0.4]{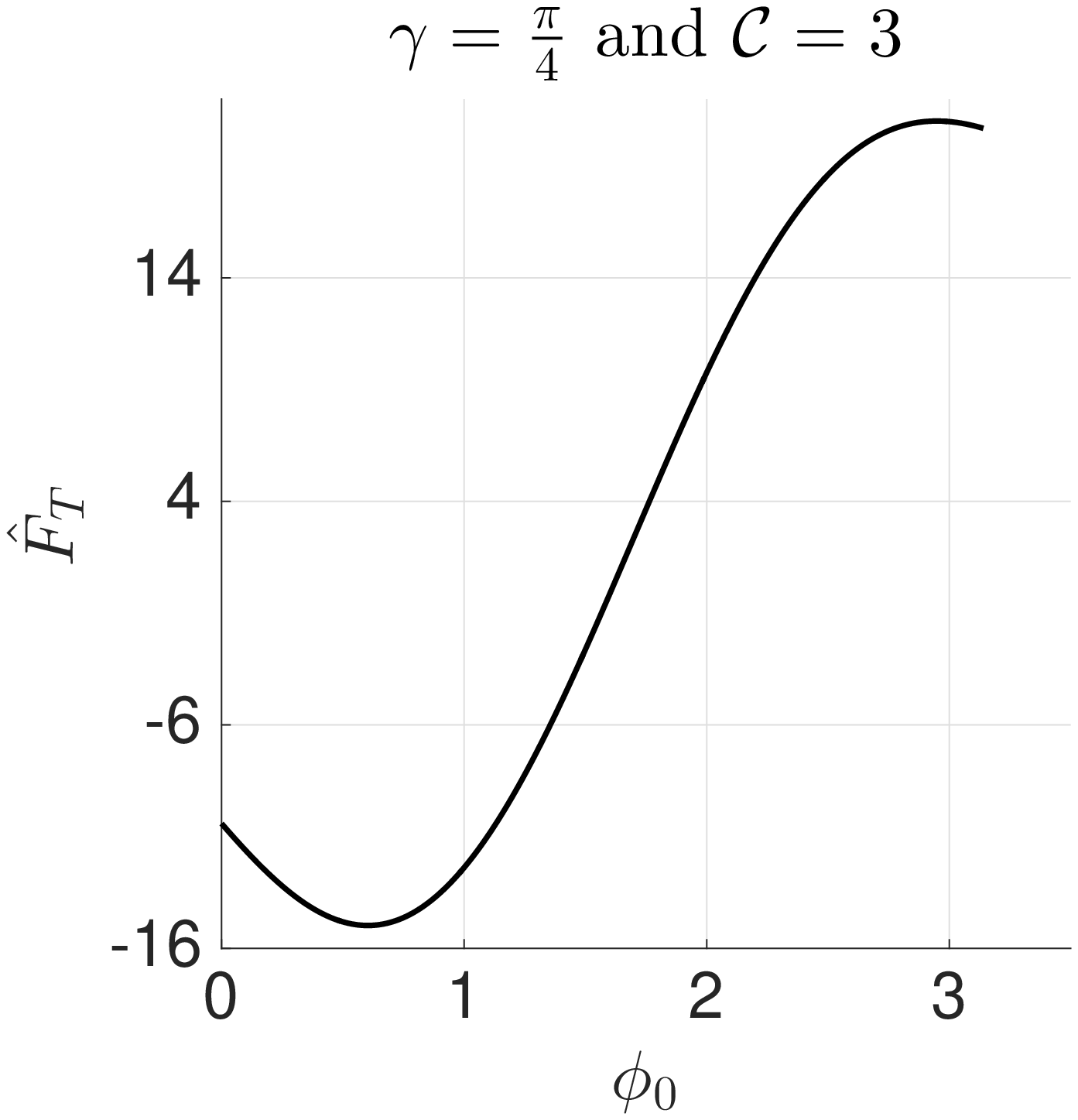}
\caption{$\hat{F}_T$ curve with parameters $\gamma = \frac{\pi}{4}$, $\mathcal{C}=3$ and $\mathcal{A}=1$ such that $\frac{d\hat{F}_T}{d\phi_0}(\pi)<0$. }
\label{case4}
\end{figure}

Based on Theorem \ref{behalthanpiby2}, there are two type of behavior $\hat{F}_T$ for $\gamma < \frac{\pi}{2}$. Two typical examples of those cases are shown in Fig.~\ref{case3} and Fig.~\ref{case4}. The following Theorem shows the number of equilibria and their stability for $\gamma < \frac{\pi}{2}$.

\begin{theorem}\label{eqstablelspiby2}
For $\gamma <\frac{\pi}{2}$, if $\frac{d\hat{F}_T}{d\phi_0}(\pi)\geq 0$, $\hat{F}_T$ admits at most one equilibrium point which is stable. If $\frac{d\hat{F}_T}{d\phi_0}(\pi)< 0$, $\hat{F}_T$ admits at most two equilibrium points, the smaller one is stable and the larger one is unstable. In addition, in the $\frac{d\hat{F}_T}{d\phi_0}(\pi)< 0$ case, if $\hat{F}_T$ admits only one equilibrium point, $\bar{\phi}_0$ is stable if $\frac{d\hat{F}_T}{d\phi_0}(\bar{\phi}_0)>0$ and it is unstable if $\frac{d\hat{F}_T}{d\phi_0}(\bar{\phi}_0)=0$.
\end{theorem}

\begin{proof}
The result comes from Eq.~(\ref{endpts}), Theorem~\ref{behalthanpiby2} and the criteria (\ref{stableeq}) and (\ref{unstableeq}). For the $\frac{d\hat{F}_T}{d\phi_0}(\pi)< 0$ case, it is similar to Theorem~\ref{eqstablepiby2}.
\end{proof}

Therefore, we conclude that, for arbitrary $\gamma$, $\hat{F}_T$ admits at most two equilibrium points, the smaller one is stable and the larger one is unstable. Moreover, if $\hat{F}_T$ has only one equilibrium point, it is unstable if $\frac{d\hat{F}_T}{d\phi_0}(\bar{\phi}_0)=0$, where $\bar{\phi}_0<\pi$, otherwise it is stable.

\begin{remark}\label{rhomlthan0}
Treinen\cite{treinen16} also studied the $\rho_m <0$ case (the density of the cylinder is less than the density of the air). If we allow $\rho_m<0$ ($\mathcal{A}<0$), the $\hat{F}_T$ curve can be shifted up, therefore, the first critical point of $\hat{F}_T$ has to be taken into consideration. Based on Theorem~\ref{behaeqpiby2}, Theorem~\ref{behagthanpiby2} and Theorem~\ref{behalthanpiby2}, if $\hat{F}_T$ has a critical point $\phi_0^{*'}$ in $\left(0,\frac{\pi}{2}\right)$, there is a possibility that $\hat{F}_T$ admits two equilibrium points, one is smaller than $\phi_0^{*'}$, the other is larger than $\phi_0^{*'}$. According to the criteria (\ref{stableeq}) and (\ref{unstableeq}), the smaller one is unstable and the larger one is stable. Thus, we agree with Treinen's conjecture $1$ for the $\rho_m<0$ case.

\end{remark}

\subsection{Asymptotic Behavior of $\mathcal{A}^*$ and $\phi_0^*$ for $\gamma = \frac{\pi}{2}$}\label{secasymp}
As discussed, $\phi_0^*$ and $\mathcal{A}^*$ in Eq.~(\ref{asymeqn}) have to be found numerically. But for $\gamma = \frac{\pi}{2}$, there exists asymptotic approximations of $\phi_0^*(\mathcal{C})$ and $\mathcal{A}^*(\mathcal{C})$ as $\mathcal{C} \rightarrow \infty$ and $\mathcal{C} \rightarrow 0$. In this section, we find these asymptotic series by applying the real analytic implicit function theorem\cite{perturbation} to the following equation:
\begin{equation}
\frac{d\hat{F}_T}{d\phi_0}(\phi_0^*(\mathcal{C});\mathcal{C}) = 0,\quad \phi_0^*(\mathcal{C})>\frac{\pi}{2}.
\label{asymeqn2}
\end{equation}
Then $\hat{F}_T\left(\phi_0^*(\mathcal{C}); \mathcal{A}^*(\mathcal{C})\right)=0$ can be solved explicitly for $\mathcal{A}^*(\mathcal{C})$. See Ref.~\onlinecite{hanzhe} for more details.

\subsubsection{As $\mathcal{C} \rightarrow 0$}

Since $\frac{d}{d\phi_0}\hat{F}_T(\pi;0)=0$ and $\frac{d^2}{d\phi_0^2}\hat{F}_T(\pi;0)\neq 0$, the implicit function theorem guarantees the existence of an analytic function $\phi_0^*(\mathcal{C})$ in terms of $\mathcal{C}$ near $\mathcal{C} = 0$ satisfying $\frac{d}{d\phi_0}\hat{F}_T\left(\phi_0^*(\mathcal{C});\mathcal{C}\right)=0$. Consider the regular asymptotic series $\phi_0^* = \pi+a_1{\mathcal{C}}+a_2{\mathcal{C}}^2+...$. We obtain 
\begin{eqnarray}
{\mathcal{A}}^* &=& \frac{2}{{\mathcal{C}}^2}+2+\pi-2\sqrt{2}{\mathcal{C}}+{\mathcal{O}}({\mathcal{C}}^2),\\
\phi_0^* &=& \pi-\sqrt{2}{\mathcal{C}}+2{\mathcal{C}}^2-\frac{7}{12}\sqrt{2}{\mathcal{C}}^3+{\mathcal{O}}({\mathcal{C}}^4).
\end{eqnarray}

\subsubsection{As $\mathcal{C} \rightarrow \infty$}

Since a regular asymptotic series doesn't work in this case, we modify the power series $\phi_0^*(\mathcal{D}) = \sum_{n=0}^{\infty}a_n\mathcal{D}^n$ to satisfy the implicit function theorem about $\mathcal{D}=0$, where $\mathcal{D} = \frac{1}{\sqrt{\mathcal{C}}}$, so that $\mathcal{D}\rightarrow 0$ as $\mathcal{C}\rightarrow \infty$. We obtain  
\begin{eqnarray}
{\mathcal{A}}^* &=& \pi+\frac{\frac{1}{3}2^{\frac{11}{4}}}{{\mathcal{C}}^{\frac{3}{2}}}+{\mathcal{O}}({\mathcal{C}}^{-\frac{1}{2}}),\\
\phi_0^* &=& \pi-\frac{2^{\frac{1}{4}}}{{\mathcal{C}}^{\frac{1}{2}}}+\frac{2^{-\frac{1}{2}}}{{\mathcal{C}}}+\frac{\frac{7}{3}2^{-\frac{13}{4}}}{{\mathcal{C}}^{\frac{3}{2}}}+{\mathcal{O}}({\mathcal{C}}^{-2}).
\end{eqnarray} 

The performance of the asymptotic series is shown in Fig.~\ref{asym}. 
\begin{figure}[h!]
\centering
     \subfloat[]{\includegraphics[scale = 0.3]{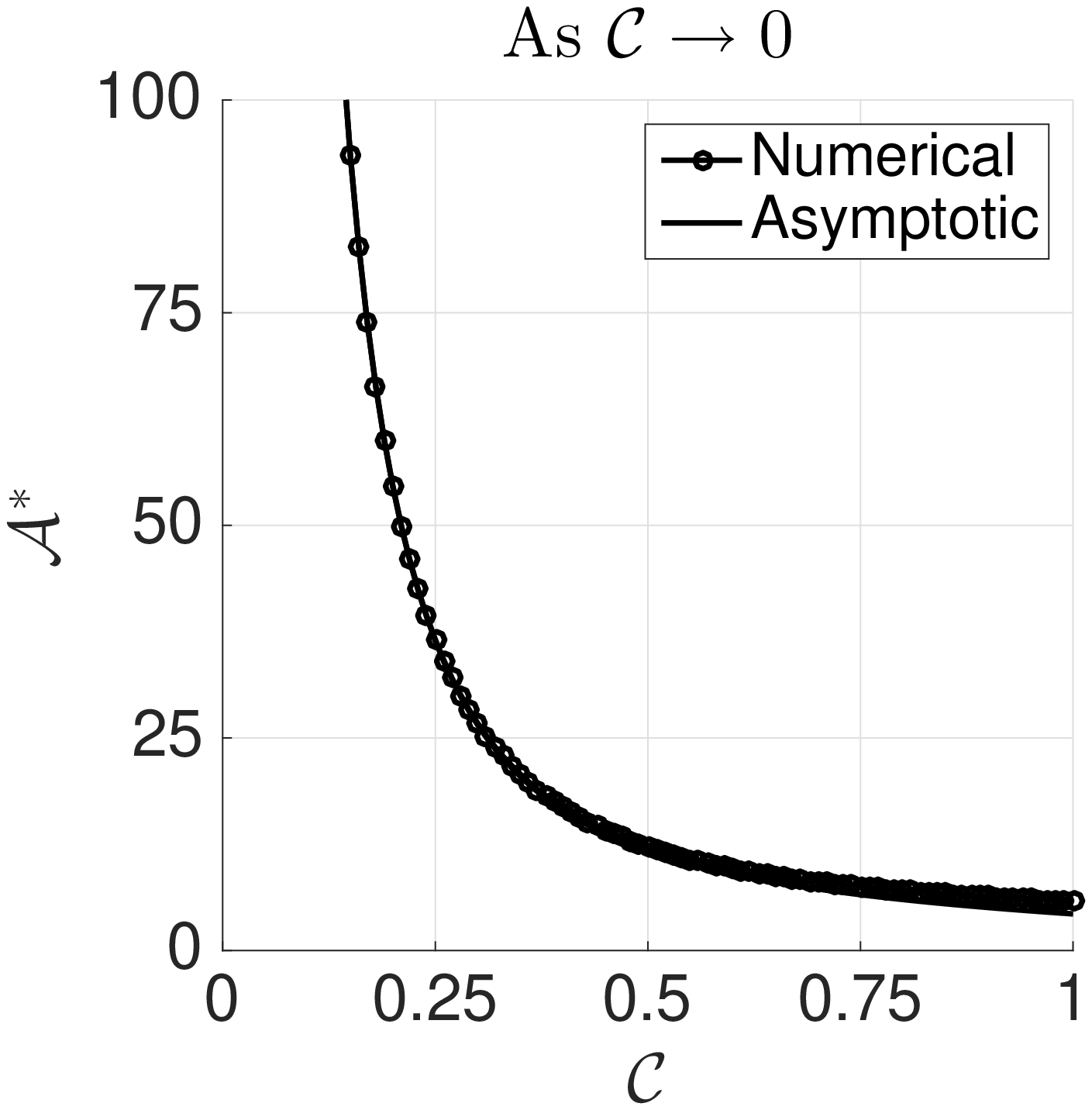}}
     \hspace{1cm}
     \subfloat[]{\includegraphics[scale = 0.3]{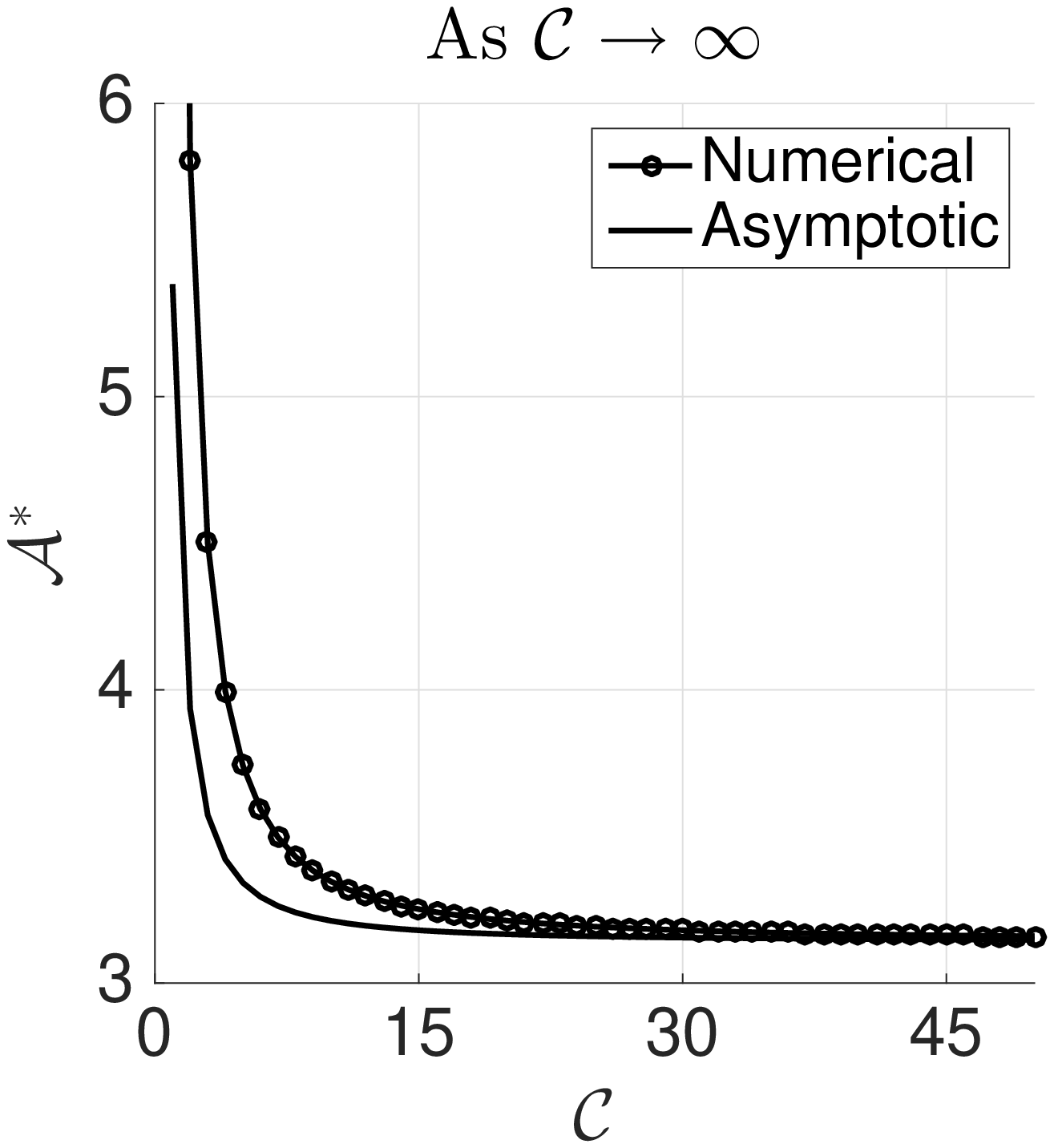}}
\caption{The performance of asymptotic series compared with numerical results.}
\label{asym}
\end{figure}

\section{Illustrating the Number of Equilibria}\label{noe}
There is a possibility that fluid interfaces on the two sides of the cylinder intersect, invalidating our model. Bhatnagar and Finn\cite{MR2259294} showed that in their example (see Remark~\ref{BFexample}) both configurations are in the non-intersection region. We plot $\mathcal{C}$ vs $\mathcal{A}$ regions with different contact angles $\gamma$ to show the number of equilibria and their validity.

\subsection{Intersection Condition}\label{intersectioncase}
We first consider the $\psi<0$ case. We find that the intersection of the fluid interfaces happens if the fluid interfaces are non-graph and the horizontal distance of the fluid interface on the right attains a non-positive value (see Fig.~\ref{inter}). We summarize the following three conditions for the intersection of the fluid interfaces.
\begin{enumerate}
\item $-\pi \leq \psi_0 \leq -\frac{\pi}{2} \quad \Leftrightarrow \quad  0\leq\phi_0+\gamma \leq \frac{\pi}{2}$.
\item $h>a \quad \Leftrightarrow \quad \cos\phi_0+\frac{2}{\mathcal{C}}\cos(\frac{\phi_0+\gamma}{2})>1$.
\item $x(-\frac{\pi}{2})\leq0 \quad \Leftrightarrow \quad \sqrt{2}+\ln\left(\tan\frac{\pi}{8}\right)-2\sin\left(\frac{\phi_0+\gamma}{2}\right)-\ln\left[-\tan\left(\frac{\phi_0+\gamma-\pi}{4}\right)\right] \geq \mathcal{C}\sin\phi_0$.
\end{enumerate}

\begin{figure}[h!]
\includegraphics[scale = 0.35]{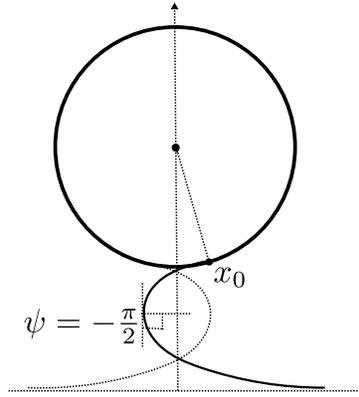}
\caption{The intersection of the fluid interfaces for the $\psi<0$ case.}
\label{inter}
\end{figure}

The conditions are similar for the $\psi>0$ case. We encode the intersection conditions by defining an intersection function
\begin{equation}
I(\phi_0,\mathcal{C})= \mathcal{C}\sin\phi_0-\sqrt{2}-\ln\left(\tan\frac{\pi}{8}\right)+2\sin\left(\frac{\phi_0+\gamma}{2}\right)+\ln\left\vert\tan\left(\frac{\phi_0+\gamma-\pi}{4}\right)\right\vert.
\label{intersectioneqn}
\end{equation}

For the $\psi<0$ case, the intersection happens when $I(\phi_0,\mathcal{C})\leq 0$ for $\gamma \in \left[0,\frac{\pi}{2}\right]$, $\phi_0 \in \left[0,\frac{\pi}{2}-\gamma\right]$. For the $\psi>0$ case, the intersection happens when $I(\phi_0,\mathcal{C})\leq 0$ for $\gamma \in \left[\frac{\pi}{2},\pi\right]$, $\phi_0 \in \left[\frac{3\pi}{2}-\gamma,\pi\right]$. In addition, $I(\phi_0,\mathcal{C})=0$ is the boundary curve between the intersection region and the non-intersection region. We have that the pair $\left(\phi_0,\mathcal{C}\right)$ lies in the non-intersection region if and only if $I(\phi_0,\mathcal{C})>0$. In Fig.~\ref{interpiby4} and Fig.~\ref{inter3piby4}, we give two examples of the intersection of fluid interfaces in the shaded region. 

We would like to know whether or not the equilibrium points lie in the intersection region. The following Theorem~\ref{eqnoninterpileqpiby2} shows the equilibrium point(s) lie in the non-intersection region for $\gamma \leq \frac{\pi}{2}$. While, for $\gamma >\frac{\pi}{2}$, Theorem~\ref{eqnonintergpiby2} the stable equilibrium point always lies in the non-intersection region. But there exists some non-physical configurations for unstable equilibrium points (see discussion of invalid equilibrium region in Sec~\ref{exfour} and Sec~\ref{exfive}).  
\begin{figure}[h!]
\includegraphics[scale = 0.4]{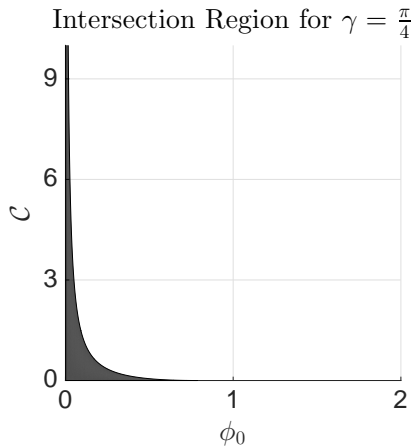}
\caption{$\gamma = \frac{\pi}{4}$, $\phi_0 \in \left[0,\frac{\pi}{4}\right]$.}
\label{interpiby4}
\end{figure}

\begin{figure}[h!]
\includegraphics[scale = 0.4]{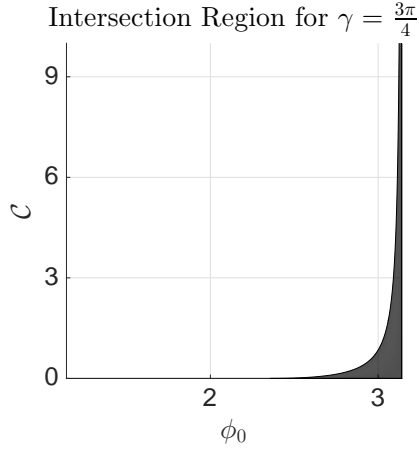}
\caption{$\gamma = \frac{3\pi}{4}$, $\phi_0 \in \left[\frac{3\pi}{4},\pi\right]$.}
\label{inter3piby4}
\end{figure}

\begin{theorem}\label{eqnoninterpileqpiby2}
The equilibrium point(s) lie in the non-intersection region for $\gamma\leq\frac{\pi}{2}$.
\end{theorem}
\begin{proof}
\begin{enumerate}
\item When $\gamma = \frac{\pi}{2}$, there is always no intersection for $\phi_0 \in (0,\pi)$ since the interface is a graph ($\phi_0=0,\pi$ are not equilibrium points, therefore they are not taken into account. 
  
\item When $\gamma <\frac{\pi}{2}$, we have the intersection happens when $I(\phi_0,\mathcal{C})\leq 0$ for $\phi_0\in \left[0,\frac{\pi}{2}-\gamma\right]$. Suppose there exist an equilibrium point $\bar{\phi}_0\in (0,\frac{\pi}{2}-\gamma]$ ($\phi_0 = 0$ is not an equilibrium point). The smallest $\bar{\phi}_0$ appears when $\mathcal{A}=0$, denoted as $\bar{\phi}_{0\min}$.
\begin{eqnarray}\label{minphi0}
\hat{F}_T(\bar{\phi}_{0\min}) &=& -2\sin(\bar{\phi}_{0\min}+\gamma)-4\mathcal{C}\cos\left(\frac{\bar{\phi}_{0\min}+\gamma}{2}\right)\sin(\bar{\phi}_{0\min})\nonumber\\
&&-\frac{1}{2}\mathcal{C}^2\sin(2\bar{\phi}_{0\min})+\mathcal{C}^2\bar{\phi}_{0\min} = 0.
\end{eqnarray}
We arrange Eq.~(\ref{minphi0}) to obtain a quadratic equation of $\mathcal{C}$.
\begin{equation}\label{quadraC1}
w\mathcal{C}^2+p\mathcal{C}+q=0,
\end{equation}
where $w=\bar{\phi}_{0\min}-\frac{1}{2}\sin(2\bar{\phi}_{0\min})>0$, $p=-4\cos\left(\frac{\bar{\phi}_{0\min}+\gamma}{2}\right)\sin(\bar{\phi}_{0\min})$ and $q = -2\sin(\bar{\phi}_{0\min}+\gamma)<0$. Thus, $\mathcal{C}$ can be solved, 

\begin{equation}\label{valc1}
\mathcal{C} = \frac{-p+\sqrt{p^2-4wq}}{2w},
\end{equation}
where $``+"$ is valid since $\mathcal{C}>0$.

We apply Eq.~({\ref{valc1}}) to the intersection function Eq.~(\ref{intersectioneqn}) and obtain
\begin{equation}\label{intercheck1}
I(\bar{\phi}_{0\min},\mathcal{C}) >0,
\end{equation}
for $\bar{\phi}_{0\min}\in (0,\frac{\pi}{2}-\gamma]$ and $\gamma \in [0,\frac{\pi}{2})$. Eq.~(\ref{intercheck1}) is evaluated numerically with Matlab. 

Thus, for $\gamma \leq \frac{\pi}{2}$, the equilibrium point(s) lie in the non-intersection region.    
\end{enumerate}
\end{proof}

\begin{theorem}\label{eqnonintergpiby2}
For $\gamma > \frac{\pi}{2}$, the stable equilibrium point always lies in the non-intersection region.
\end{theorem}
\begin{proof}
For $\gamma \in (\frac{\pi}{2},\pi]$, the intersection happens when $I(\phi_0,\mathcal{C})\leq 0, \phi_0 \in \left[\frac{3\pi}{2}-\gamma,\pi\right]$. By Theorem~\ref{behagthanpiby2}, $\hat{F}_T$ has a critical point $\phi^*_0>\frac{\pi}{2}$, thus 
\begin{eqnarray}\label{totalcri}
\frac{d\hat{F}_T}{d\phi_0}(\phi^*_0) &=& \mathcal{C}\left(2\sin\left(\frac{\phi^*_0+\gamma}{2}\right)\sin\phi^*_0-4\cos\left(\frac{\phi^*_0+\gamma}{2}\right)\cos\phi^*_0\right) \nonumber \\
&&+(1-\cos2\phi^*_0)\mathcal{C}^2-2\cos\gamma\cos\phi^*_0+2\sin\gamma\sin\phi^*_0 = 0.
\end{eqnarray}

We arrange Eq.~(\ref{totalcri}) to obtain a quadratic equation of $\mathcal{C}$.
\begin{equation}\label{quadra2}
\bar{w} \mathcal{C}^2 + \bar{p} \mathcal{C}+\bar{q} = 0,
\end{equation}
where $\bar{w}=1-\cos2\phi^*_0>0$, $\bar{q}=-2\cos(\phi^*_0+\gamma)<0$ for $\phi^*_0\in (\frac{3\pi}{2}-\gamma,\pi)$, $\gamma \in (\frac{\pi}{2},\pi]$ and $\bar{p} = 2\sin\left(\frac{\phi^*_0+\gamma}{2}\right)\sin\phi^*_0-4\cos\left(\frac{\phi^*_0+\gamma}{2}\right)\cos\phi^*_0$. We are going to check whether the critical point $\phi^*_0$ lies in the intersection region. If we can show $I(\phi^*_0,\mathcal{C})>0$ for $\phi^*_0\in (\frac{3\pi}{2}-\gamma,\pi)$, $\gamma \in (\frac{\pi}{2},\pi]$, we can conclude the stable equilibrium point lies in the non-intersection region. This follows since $\frac{\pi}{2}\leq \phi_{01}<\phi^*_0$ where $\phi_{01}$ is the stable equilibrium point and $I(\phi_0,\mathcal{C})$ is strictly decreasing on $\phi_0$ for $\phi_0\in [\frac{3\pi}{2}-\gamma,\pi)$, $\gamma \geq \frac{\pi}{2}$. This implies $I(\phi_{01},\mathcal{C})>I(\phi^*_0,\mathcal{C})>0$.

From Eq.~(\ref{quadra2}), we have
\begin{equation}\label{valc2}
\mathcal{C} = \frac{-\bar{p}+\sqrt{\bar{p}^2-4\bar{w}\bar{q}}}{2\bar{w}},
\end{equation}
where $``+"$ is valid, since $\mathcal{C}>0$.

We apply Eq.~({\ref{valc2}}) to the intersection function Eq.~(\ref{intersectioneqn}) and obtain
\begin{equation}\label{intercheck2}
I(\phi^*_0,\mathcal{C}) >0,
\end{equation}
for $\phi^*_0\in (\frac{3\pi}{2}-\gamma,\pi)$, $\gamma \in (\frac{\pi}{2},\pi]$. Eq.~(\ref{intercheck2}) is evaluated numerically with Matlab. 

Thus, the stable equilibrium point always lies in the non-intersection region.

\end{proof}


\subsection{$\mathcal{C}$ vs $\mathcal{A}$: Regions with Different Numbers of Equilibria}\label{AC}

We study the number of equilibria in Sec.~\ref{stabe}. While in this section, we would like to discuss the number of equilibria in consideration of the intersection condition. Since $\mathcal{A}$, $\mathcal{C}$ and contact angle $\gamma$ will affect the number of equilibria of our floating system, the $\mathcal{C}$ vs $\mathcal{A}$ region will be helpful and clear.  Examples with typical contact angles (e.g. $\gamma = 0, \frac{\pi}{4}, \frac{\pi}{2}, \frac{3\pi}{4}$ and $\pi$) will be given. In the $\mathcal{C}$ vs $\mathcal{A}$ plane, we define $\mathcal{C}_i$ as the boundary curves between the regions with different number of equilibria. According to the discussion of the behavior of the $\hat{F}_T$ curve in Sec.~\ref{stabe}, the sign of $\hat{F}_T(\pi)$ and the sign of $\hat{F}_T(\phi_0^*)$, $\phi_0^*>\frac{\pi}{2}$, play important roles in determining the number of equilibria. An equilibrium $\bar{\phi}_0$ is in the non-intersection region when $I(\bar{\phi}_0,\mathcal{C})>0$. The boundary curves $\mathcal{C}_i$ can be expressed as follows
\begin{enumerate}
\item $\mathcal{C}_1: \hat{F}_T(\pi)=0 \quad \Leftrightarrow \quad  (\mathcal{A}-\pi)\mathcal{C}^2 = 2\sin\gamma$. 
\item $\mathcal{C}_2: \hat{F}_T(\phi_0^*)=0$, where $\phi_0^* >\frac{\pi}{2}$ satisfying $\frac{d \hat{F}_T}{d \phi_0}(\phi_0^*)=0$. 
\item $\mathcal{C}_3: \hat{F}_T(\bar{\phi}_{02})=0$, where $\bar{\phi}_{02}$ satisfies $I(\bar{\phi}_{02},\mathcal{C})=0$.
\end{enumerate}

Only the curve $\mathcal{C}_1$ can be solved for analytically, that is, $\mathcal{C}_1(\mathcal{A})=\sqrt{\frac{2\sin\gamma}{\mathcal{A}-\pi}}$ (for $\gamma \neq 0, \pi$). While, the critical point of $\frac{d\hat{F}_T}{d\phi_0}$, $\phi_0^*$ and the angle $\bar{\phi}_{02}$ have to be solved for numerically. In the following examples, we will analyze the boundary curves $\mathcal{C}_i$ and plot the $\mathcal{C}$ vs $\mathcal{A}$ regions. 

\subsubsection{Example One: $\gamma = 0$}\label{gamma0}
When $\gamma = 0$, we have $\frac{d\hat{F}_T}{d\phi_0}(\pi)=2>0$ so that $\hat{F}_T$ has at most one equilibrium point, denoted as  $\bar{\phi}_0$ if it exists. Therefore, $\mathcal{C}_1$, the boundary between the zero equilibrium point region and the one equilibrium point region, has the following expression
\begin{equation}
\hat{F}_T(\pi) = 0 \quad \Leftrightarrow \quad \mathcal{A} = \pi.
\end{equation}

In Fig.~\ref{ac0}, the one equilibrium point region is to the left of $\mathcal{C}_1$ and  the no equilibrium point region is to the right of $\mathcal{C}_1$. By Theorem~\ref{eqnoninterpileqpiby2}, the equilibrium point never lies in the intersection region. 

\begin{figure}[h!]
\includegraphics[scale = 0.4]{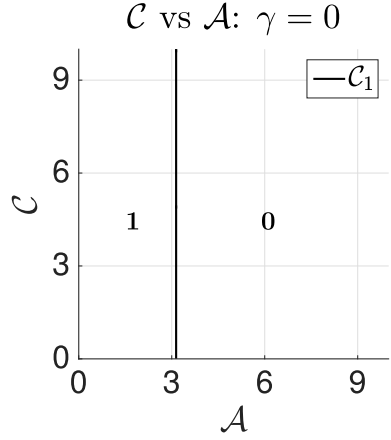}
\caption{$\mathcal{C}$ vs $\mathcal{A}$: $\gamma = 0$. $\mathbf{0}$ indicates the zero equilibrium point region and $\mathbf{1}$ indicates the one equilibrium point region. The boundary curve between region $\mathbf{0}$ and region $\mathbf{1}$ is $\mathcal{A}=\pi$.}
\label{ac0}
\end{figure}

\subsubsection{Example Two: $\gamma = \frac{\pi}{4}$}

When $\gamma = \frac{\pi}{4}$, $\sign\left(\frac{d\hat{F}_T}{d\phi_0}(\pi)\right)$ can be either nonnegative or negative. We have the following two cases: 
\begin{enumerate}
\item $\frac{d\hat{F}_T}{d\phi_0}(\pi)\geq 0 \quad \Leftrightarrow \quad -4\mathcal{C}\sin\frac{\gamma}{2}+2\cos\gamma\geq 0$.

The inequality above implies $\mathcal{C}\in [0,\mathcal{C}_0]$ where $\mathcal{C}_0=\frac{\cos\gamma}{2\sin\frac{\gamma}{2}}$. In this case, $\hat{F}_T$ has at most one equilibrium point, denoted as $\bar{\phi}_0$ if it exists.

When $\mathcal{C}\in [0,\mathcal{C}_0]$, the boundary curve $\mathcal{C}_{11}(\mathcal{A})$ is 
\begin{equation}
\hat{F}_T(\pi) = 0 \quad  \Leftrightarrow \quad \mathcal{C}_{11}(\mathcal{A})=\sqrt{\frac{\sqrt{2}}{\mathcal{A}-\pi}},
\label{c11}
\end{equation}
where $\mathcal{A}\in [\mathcal{A}_0,\infty)$ and $\mathcal{A}_0$ satisfies $\mathcal{C}_{11}(\mathcal{A}_0)=\mathcal{C}_0$. Moreover, $\mathcal{C}_{11}(\mathcal{A})$ is the boundary curve between the zero equilibrium point region and the one equilibrium point region. The zero equilibrium point region is above $\mathcal{C}_{11}(\mathcal{A})$ and the one equilibrium point region is below $\mathcal{C}_{11}(\mathcal{A})$. 

In this case, there is no $\phi_0^*>\frac{\pi}{2}$. Thus, $\mathcal{C}_2$ curve is not defined on $\mathcal{A}\in[\mathcal{A}_0,\infty)$.

\item $\frac{d\hat{F}_T}{d\phi_0}(\pi)<0 \quad \Leftrightarrow \quad -4\mathcal{C}\sin\frac{\gamma}{2}+2\cos\gamma < 0$.

$\frac{d\hat{F}_T}{d\phi_0}(\pi)< 0$ implies $\mathcal{C}>\mathcal{C}_0$. In this case, $\hat{F}_T$ has at most two equilibrium points, denoted as $\bar{\phi}_{01}$ and $\bar{\phi}_{02}$ if they exist.

When $\mathcal{C}>\mathcal{C}_0$, the boundary curve $\mathcal{C}_{12}(\mathcal{A})$ is

\begin{equation}
\hat{F}_T(\pi) = 0 \quad \Leftrightarrow \quad\mathcal{C}_{12}(\mathcal{A})=\sqrt{\frac{\sqrt{2}}{\mathcal{A}-\pi}}\quad \textrm{where} \quad \mathcal{A}\in (\pi,\mathcal{A}_0].
\label{c12}
\end{equation}

Since $\frac{d\hat{F}_T}{d\phi_0}(\pi)<0$, $\mathcal{C}_{12}(\mathcal{A})$ is the boundary curve between the one equilibrium point region and the two equilibrium points region. The one equilibrium point region is to the left of $\mathcal{C}_{12}(\mathcal{A})$ and the two equilibrium points region is to the right of $\mathcal{C}_{12}(\mathcal{A})$. Moreover, $\mathcal{C}_{11}(\mathcal{A})$ in (\ref{c11}) and $\mathcal{C}_{12}(\mathcal{A})$ in (\ref{c12}) can be combined together, denoted by $\mathcal{C}_{1}(\mathcal{A})$
\begin{equation}
\mathcal{C}_{1}(\mathcal{A}) = \sqrt{\frac{\sqrt{2}}{\mathcal{A}-\pi}}\quad \textrm{where} \quad \mathcal{A}\in(\pi,\infty).
\end{equation}

In this case, the critical point $\phi_0^*>\frac{\pi}{2}$ exists and $\mathcal{C}_2(\mathcal{A})$ can be obtained numerically on $\mathcal{A}\in(\pi,\mathcal{A}_0]$.
\end{enumerate}

In Fig.~\ref{acpiby4}, the region between the curve $\mathcal{C}_{1}(\mathcal{A})$ and $\mathcal{C}_{2}(\mathcal{A})$ is the two equilibrium points region with $\mathcal{A} \in (\pi, \mathcal{A}_0]$. The one equilibrium region is below $\mathcal{C}_{1}(\mathcal{A})$ and the zero equilibrium region is above $\mathcal{C}_{12}(\mathcal{A})$ and $\mathcal{C}_{2}(\mathcal{A})$ curves. The equilibrium point(s) do not lie in the intersection region by Theorem~\ref{eqnoninterpileqpiby2}.
\begin{figure*}[h!]
\includegraphics[scale = 0.4]{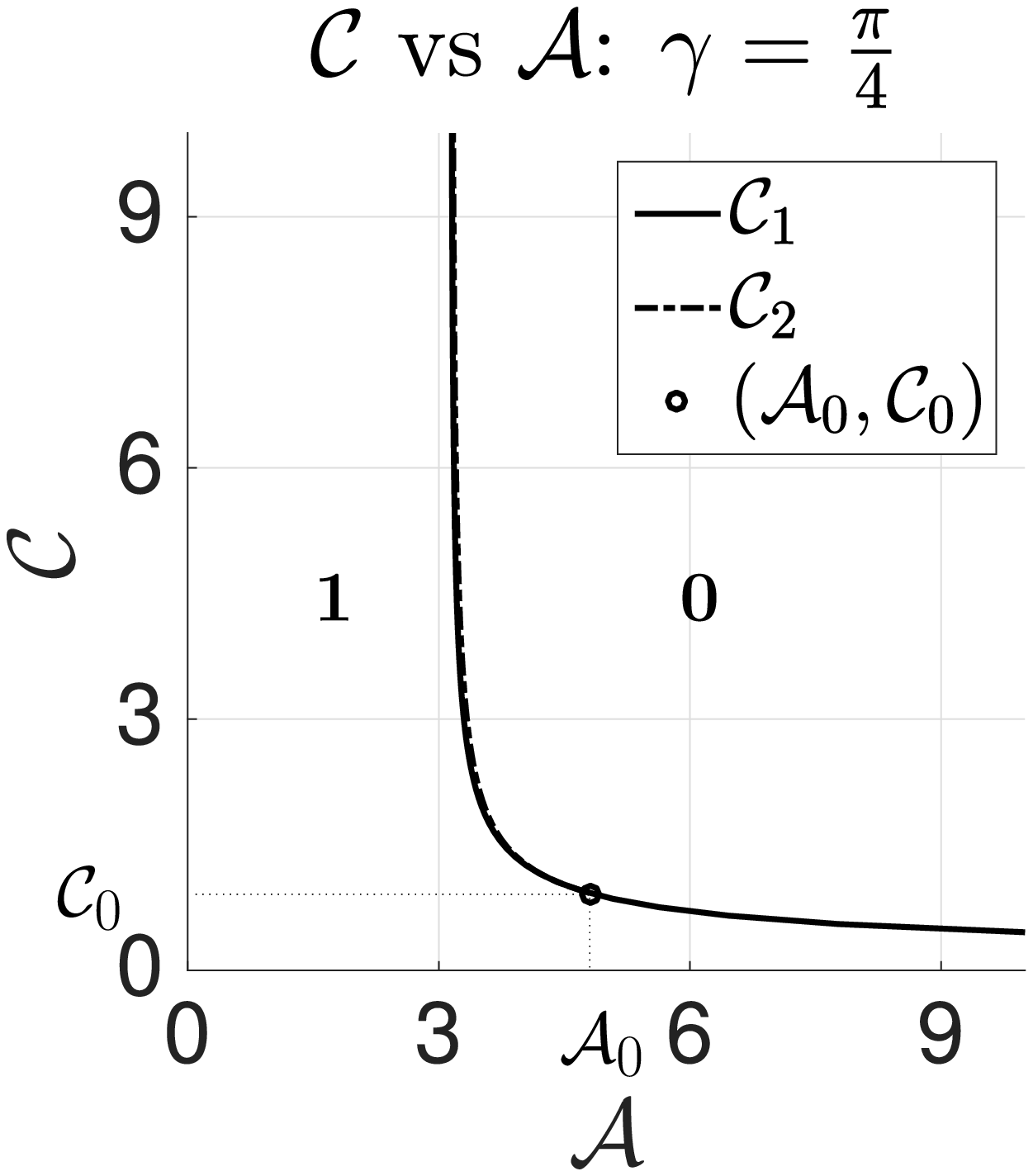}
\includegraphics[scale = 0.4]{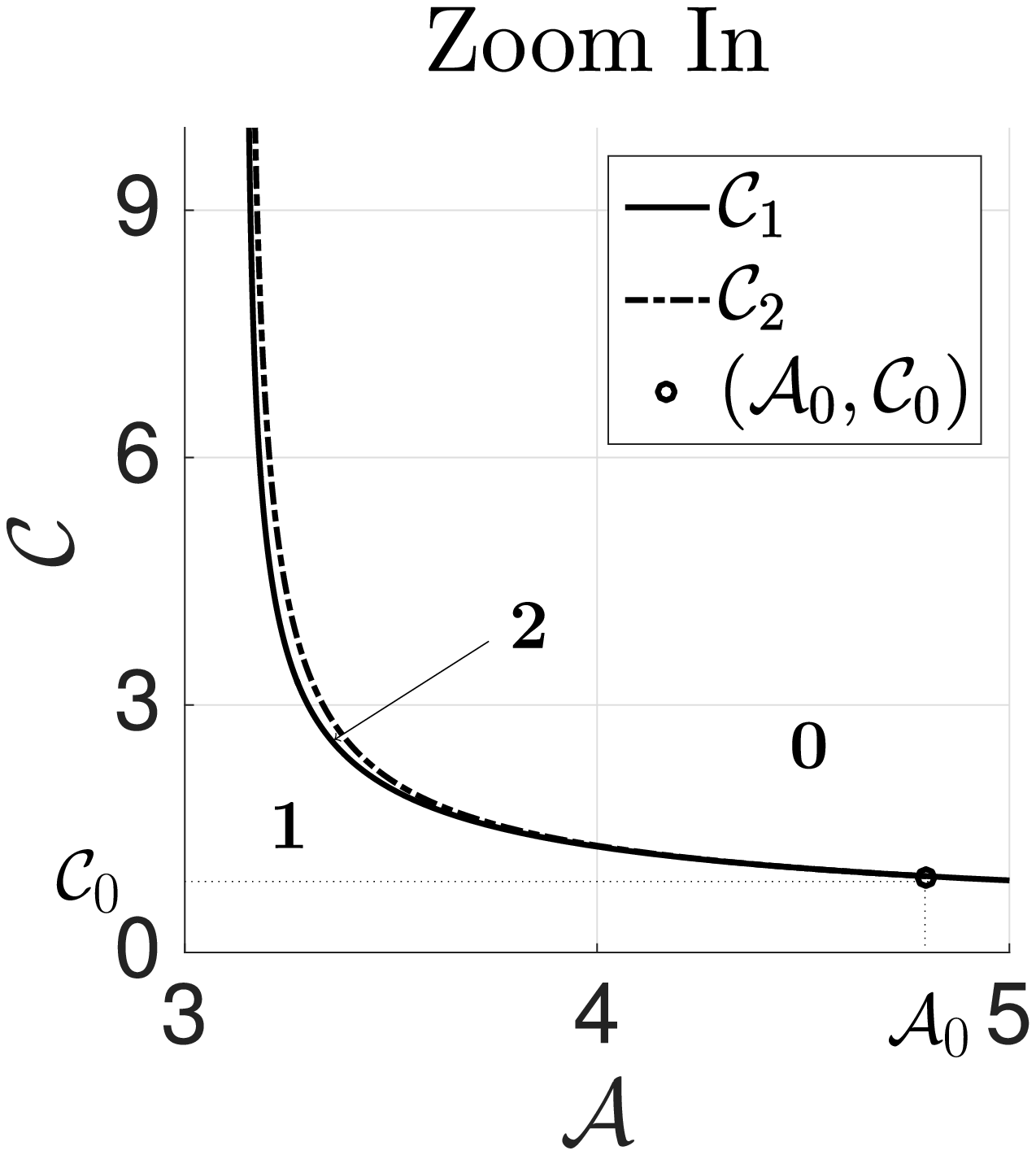}
\caption{$\mathcal{C}$ vs $\mathcal{A}$: $\gamma = \frac{\pi}{4}$. $\mathbf{0}$ indicates the zero equilibrium point region, $\mathbf{1}$ indicates the one equilibrium point region and $\mathbf{2}$ indicates the two equilibrium points region. The boundary curve ends at point $(\mathcal{A}_0,\mathcal{C}_0) = \left(\pi+\frac{4\sqrt{2}}{2+\sqrt{2}}, \frac{\sqrt{2+\sqrt{2}}}{2}\right)$.}
\label{acpiby4}
\end{figure*}

\subsubsection{Example Three: $\gamma = \frac{\pi}{2}$}

When $\gamma = \frac{\pi}{2}$, the intersection of the fluid interfaces never happens by Theorem~\ref{eqnoninterpileqpiby2}. We have the explicit expression for the boundary curve $\mathcal{C}_1(\mathcal{A})=\sqrt{\frac{2}{\mathcal{A}-\pi}}$, and the boundary curve $\mathcal{C}_2(\mathcal{A})$ can be obtained numerically. It is the inverse of $\mathcal{A}^*(\mathcal{C})$ (replacing $\mathcal{A}^*$ by $\mathcal{A}$). In addition, we have discussed the asymptotic series of $\mathcal{A}^*$ for both $\mathcal{C} \rightarrow 0$ and $\mathcal{C} \rightarrow \infty$ in Sec.~\ref{secasymp}. This gives the $\gamma = \frac{\pi}{2}$ case as shown in Fig.~\ref{acpiby2}. The zero equilibrium point region is above $\mathcal{C}_{2}(\mathcal{A})$, the two equilibrium points region is between $\mathcal{C}_{1}(\mathcal{A})$ and $\mathcal{C}_{2}(\mathcal{A})$. The one equilibrium point region is below $\mathcal{C}_{1}(\mathcal{A})$. The stability of the equilibrium point(s) can be seen in Theorem~\ref{eqstablepiby2} or in Table~\ref{eqstabletb}.

\begin{figure}[h!]
\includegraphics[scale = 0.4]{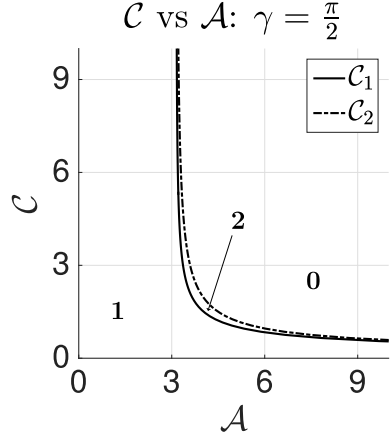}
\caption{$\mathcal{C}$ vs $\mathcal{A}$: $\gamma = \frac{\pi}{2}$. $\mathbf{0}$ indicates the zero equilibrium point region, $\mathbf{1}$ indicates the one equilibrium point region and $\mathbf{2}$ indicates the two equilibrium points region.}
\label{acpiby2}
\end{figure}

\subsubsection{Example Four: $\gamma = \frac{3\pi}{4}$}\label{exfour}

When $\gamma = \frac{3\pi}{4}$, $\hat{F}_T$ can admit at most two equilibria, denoted as $\bar{\phi}_{01}$ and $\bar{\phi}_{02}$ if they exist. By Theorem~\ref{eqnonintergpiby2}, $\bar{\phi}_{01}$ never lies in the intersection region. But for $\bar{\phi}_{02} \in \left[\frac{3\pi}{4},\pi\right]$, $I(\phi_0,\mathcal{C})$ is needed to test their validity. Therefore $\mathcal{C}_3(\mathcal{A})$ is the boundary curve between the one valid and one invalid equilibrium point region and the two (valid) equilibrium points region. 

If $\hat{F}_T(\pi)>0$, $\hat{F}_T$ admits exactly one equilibrium point $\bar{\phi}_0$. Since $I(\bar{\phi}_0,\mathcal{C})>0$, the equilibrium point never lies in intersection region. When $\hat{F}_T(\pi)=0$, $\bar{\phi}_0 = \pi$ is also an equilibrium point, but it's invalid 
(since $I(\pi,\mathcal{C})<0$). Therefore, $\mathcal{C}_1(\mathcal{A})$ is the boundary curve between the one equilibrium point region and the one valid, one invalid equilibrium point region. Explicitly, we have the form $\mathcal{C}_1(\mathcal{A})=\sqrt{\frac{\sqrt{2}}{\mathcal{A}-\pi}}$.

In Fig.~\ref{ac3piby4}, the one equilibrium point region is below $\mathcal{C}_1(\mathcal{A})$, the zero equilibrium point region is above $\mathcal{C}_2(\mathcal{A})$. The one valid, one invalid equilibrium point region is bounded by $\mathcal{C}_1(\mathcal{A})$ and $\mathcal{C}_3(\mathcal{A})$. The two (valid) equilibrium points region is bounded by $\mathcal{C}_2(\mathcal{A})$ and $\mathcal{C}_3(\mathcal{A})$. 
\begin{figure}[h!]
\includegraphics[scale = 0.4]{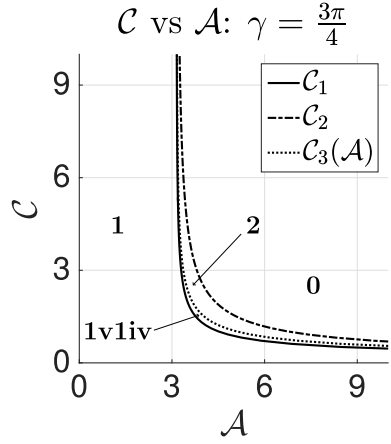}
\caption{$\mathcal{C}$ vs $\mathcal{A}$: $\gamma = \frac{3\pi}{4}$. $\mathbf{0}$ indicates the zero equilibrium point region, $\mathbf{1}$ indicates the one equilibrium point region, $\mathbf{1v1iv}$ indicates the one valid, one invalid equilibrium point region and $\mathbf{2}$ indicates the two (valid) equilibrium points region.}
\label{ac3piby4}
\end{figure}

\subsubsection{Example Five: $\gamma = \pi$}\label{exfive}

When $\gamma = \pi$, the results are similar to the previous case, $\gamma =\frac{3\pi}{4}$. The same strategy can be applied to obtain $\mathcal{C}_2(\mathcal{A})$ and $\mathcal{C}_3(\mathcal{A})$. The only difference is that the boundary curve $\mathcal{C}_1$ between the one equilibrium point region and the one valid, one invalid equilibrium point region is $\mathcal{A}=\pi$ (see Fig.~\ref{acpi}).

\begin{figure}[h!]
\includegraphics[scale = 0.4]{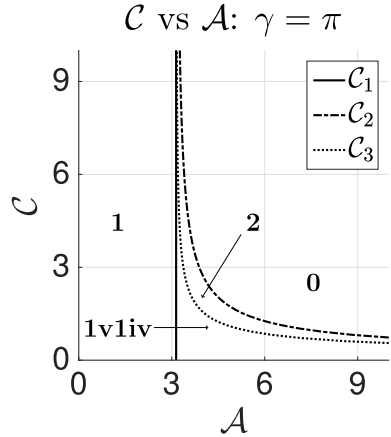}
\caption{$\mathcal{C}$ vs $\mathcal{A}$: $\gamma = \pi$. $\mathbf{0}$ indicates the zero equilibrium point region, $\mathbf{1}$ indicates the one equilibrium point region, $\mathbf{1v1iv}$ indicates the one valid, one invalid equilibrium point region and $\mathbf{2}$ indicates the two (valid) equilibrium points region. The $\mathcal{C}_1$ curve is $\mathcal{A}=\pi$.}
\label{acpi}
\end{figure}

\begin{remark}\label{Alesspi}
If $0<\mathcal{A}<\pi$, from Eqs.(\ref{f0value}) and (\ref{fpivalue}), we have $\hat{F}_T(\pi) = 2\sin\gamma +\mathcal{C}^2(\pi-\mathcal{A})>0$ and $\hat{F}_T(0)=-\mathcal{A}\mathcal{C}^2-2\sin\gamma <0$ for arbitrary $\gamma$. From Theorem \ref{eqstablepiby2}, Theorem \ref{eqstablegtpiby2} and Theorem \ref{eqstablelspiby2}, $\hat{F}_T$ admits only one equilibrium point, which is stable. The $\mathcal{A} = \pi$ case is analogous to the $0<\mathcal{A}<\pi$ case, except when $\gamma = 0$ or $\pi$. When $\mathcal{A}=\pi$ and $\gamma = 0$, $\hat{F}_T(\pi) = 0$ and $\frac{d\hat{F}_T}{d\phi_0}(\pi)>0$, therefore, from Theorem~\ref{eqstablelspiby2}, $\phi_0 = \pi$ is the only equilibrium point, which is stable. But for $\mathcal{A}=\pi$ and $\gamma = \pi$, $\hat{F}_T(\pi)=0$, $I(\pi,\mathcal{C}) <0$, from Theorem~\ref{eqstablegtpiby2}, there are two equilibrium points. But $\phi_0 = \pi$ lies in the intersection region, which is an invalid equilibrium point (see FIGs~\ref{ac0}, \ref{acpiby4}, \ref{acpiby2}, \ref{ac3piby4} and \ref{acpi}). 
\end{remark}

\subsection{An Example That Admits Two Configurations}
In this section, we give an example that admits two configurations. With contact angle $\gamma = \frac{\pi}{2}$, $\mathcal{A} = 3.8$ and $\mathcal{C} = 2$, total force curve can be shown in Fig.~\ref{multi}. The corresponding two equilibrium points are $\bar{\phi}_{01}=2.3915$ and $\bar{\phi}_{02}=3.0178$. Based on Theorem~\ref{eqstablepiby2}, the configuration in Figure~\ref{smaller} is stable and the configuration in Figure~\ref{larger} is unstable.

\begin{figure}[h]
\includegraphics[scale = 0.4]{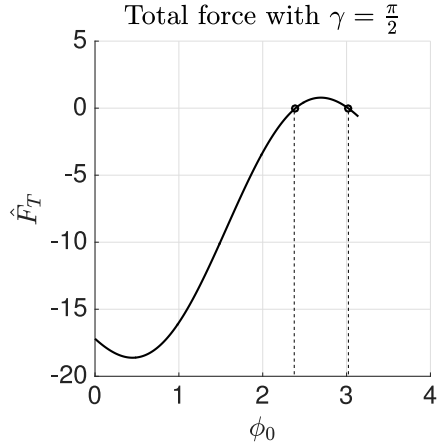}
\caption{Two Equilibrium points: $\bar{\phi}_{01} = 2.3915$ and $\bar{\phi}_{02} = 3.0178$.}
\label{multi}
\end{figure}
\begin{figure}[h]
\centering 
\subfloat[]{\includegraphics[scale = 0.4]{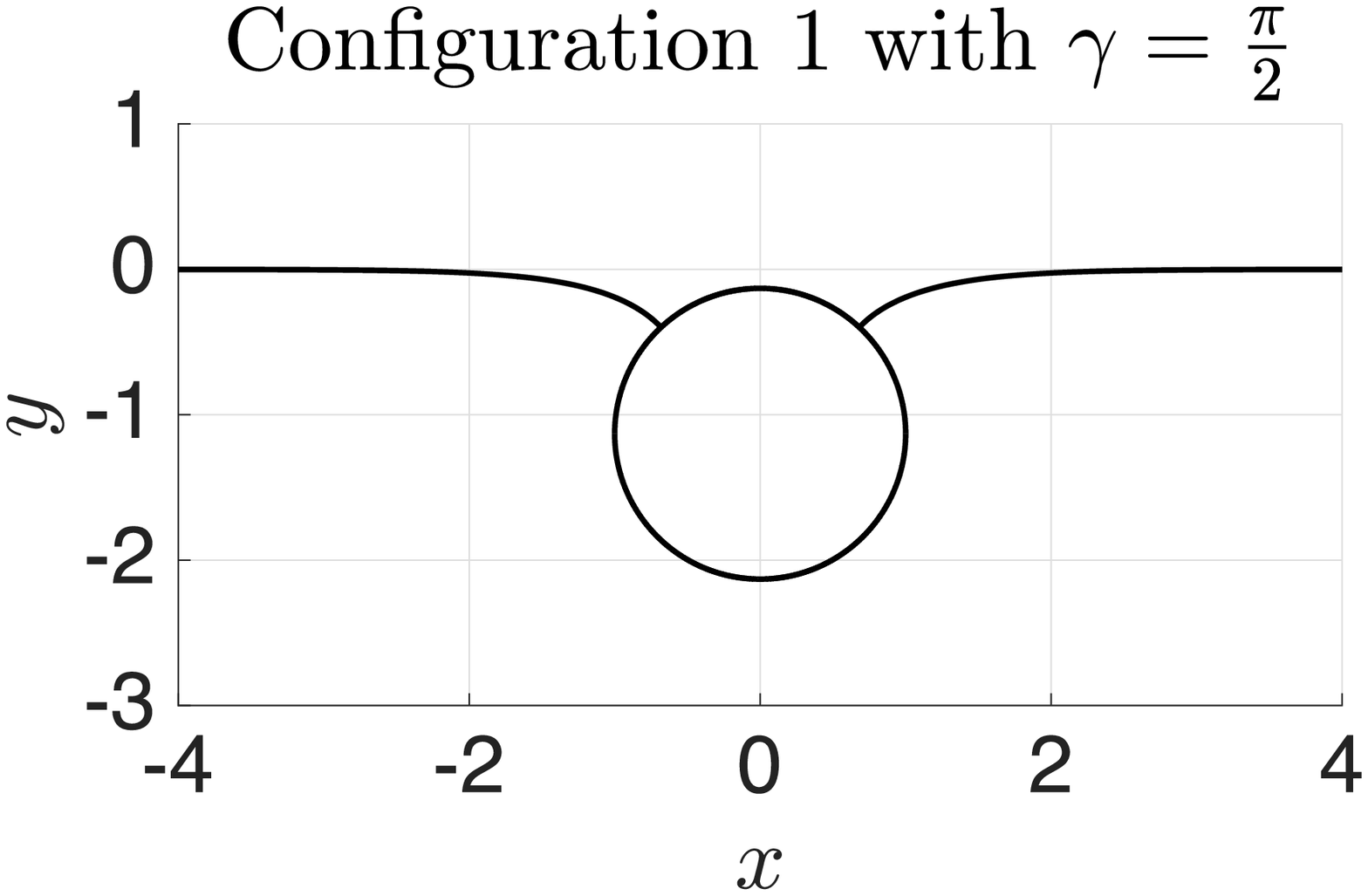}\label{smaller}}
\hspace{1cm}
\subfloat[]{\includegraphics[scale = 0.4]{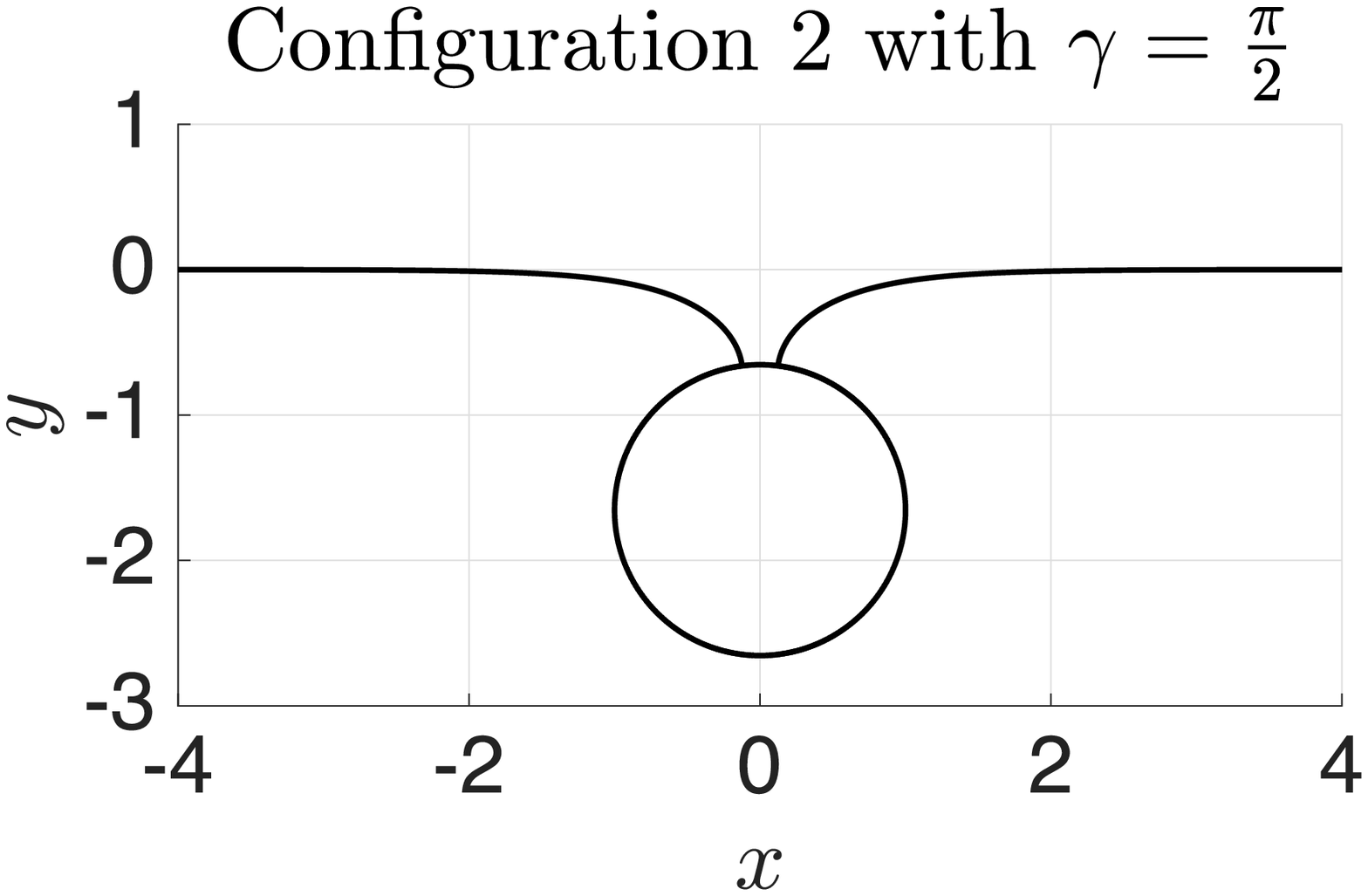}\label{larger}}
\caption{Parameters: $\gamma = \frac{\pi}{2}$, $\mathcal{A}=3.8$, $\mathcal{C}=2$ and radius $a = 1$.}
\end{figure}

\section{Conclusion and Related Work}
We have studied the floating configurations and their stability of a horizontal cylinder on an infinite reservoir. Two new elements were found: 1) the relation (\ref{detdh}) between the relative total energy $E_T$ and the total force $F_T$, which was used for determining the stability behavior of the equilibria; 2) the limitation of the intersection of fluid interfaces, determined by the intersection function $I(\phi_0,\mathcal{C})$.

Based on $-\frac{d{E_T}}{dh}=F_T$, the sign equivalence, $\sign(\frac{d{E_T}}{d\phi_0}) = \sign(F_T)$ and $\sign\left(\frac{d^2 E_T}{d \phi_0^2}(\bar{\phi}_0)\right)= \sign \left(\frac{dF_T}{d\phi_0}(\bar{\phi}_0)\right)$, gives a convenient way to minimize $E_T$. $\frac{dF_T}{d\phi_0}(\bar{\phi}_0)>0$ implies the local minimum of $E_T$ of the equilibrium point $\bar{\phi}_0$. Based on Theorem~\ref{eqstablepiby2}, Theorem~\ref{eqstablegtpiby2} and Theorem~\ref{eqstablelspiby2}, the numbers of equilibria and their stability can be classified as follows.
\begin{enumerate}
\item When $\gamma \geq \frac{\pi}{2}$, $\hat{F}_T$ admits at most two equilibrium points $\bar{\phi}_{01}$ and $\bar{\phi}_{02}$, the smaller equilibrium point $\bar{\phi}_{01}$ is stable and the larger equilibrium point $\bar{\phi}_{02}$ is unstable. In addition, if $\hat{F}_T$ admits only one equilibrium point $\bar{\phi}_0$ and $\frac{d\hat{F}_T}{d\phi_0}(\bar{\phi}_0)>0$, it is stable. If $\frac{d\hat{F}_T}{d\phi_0}(\bar{\phi}_0)=0$, it is unstable. 
\item When $\gamma < \frac{\pi}{2}$, if $\frac{d\hat{F}_T}{d\phi_0}(\pi)<0$, $\hat{F}_T$ behaves the same as $\hat{F}_T$ with $\gamma \geq \frac{\pi}{2}$. If $\frac{d\hat{F}_T}{d\phi_0}(\pi) \geq 0$, $\hat{F}_T$ admits at most one equilibrium point $\bar{\phi}_0$ which is stable. 
\end{enumerate}

In the analysis of forces, we assume the surface tension force $F_\sigma$ exists only along the fluid interface, which contradicts Young's diagram. While the relation $-\frac{d{E_T}}{dh}=F_T$ implicitly supports Finn's assertion\cite{MR2259294, youngpara}.

When $\gamma = \frac{\pi}{2}$, there is always no intersection of the fluid interfaces, and when $\gamma \neq \frac{\pi}{2}$, intersection may occur. For arbitrary contact angle $\gamma\leq \frac{\pi}{2}$, there is no equilibrium point lying in the intersection region (see Theorem~\ref{eqnoninterpileqpiby2}). For $\gamma >\frac{\pi}{2}$, we observe the one valid, one invalid equilibrium point region exists and there always is a stable equilibrium in the non-intersection region (see Theorem~\ref{eqnonintergpiby2}). Considering the intersection region, we illustrate the numbers of equilibria and their stability behavior in $\mathcal{A}\mathcal{C}$ plane. For the cases $\gamma = 0,\frac{\pi}{4},\frac{\pi}{2},\frac{3\pi}{4},\pi$, we discuss the boundary curves between the regions with different numbers of equilibria. 

Treinen\cite{treinen16} also studied the unbounded horizontal cylinder problem for both the $\rho_m<0$ and the $\rho_m>\rho$ cases. In the $\rho_m>\rho$ ($\mathcal{A}>\pi$) case, our study agrees with Treinen's conjecture 1, the system admits at most two equilibrium points, the smaller one is stable and the larger one is unstable. The discussion of the $\rho_m<0$ case can be seen in Remark~\ref{rhomlthan0}. Moreover, for the $0<\rho_m<\rho$ ($0<\mathcal{A}<\pi$) case, there is only one configuration, which is stable. For the $\rho_0=\rho$ ($\mathcal{A}=\pi$) case, it is analogous to the $0<\rho_m<\rho$ case except for $\gamma = \pi$. When $\gamma = \pi$, there are two equilibrium points, but the larger one $\bar{\phi}_0 = \pi$ is not physically realizable (see Remark~\ref{Alesspi}).

The horizontal cylinder behaves different in a laterally finite container than in the unbounded reservoir. McCuan and Treinen\cite{mccuan_treinen17} studied the laterally finite container case and gave an example that there are three equilibrium points and two of them are stable.

A ball floating on an unbounded bath deserves study. There is a non-monotone relation between $h$ and $\phi_0$ (see Ref.~\onlinecite{hanzhe}), which makes this problem significantly different than the cylinder floating on an unbounded bath.

\appendix
\section{Computation of the Total Energy $E_T$}\label{Apx1}
In this section, the detailed derivation of both surface tension energy $E_\sigma$ and the fluid potential energy $E_{F}$ are given.  

\subsection{Surface Tension Energy $E_\sigma$}
When the fluid interface is a graph, we have discussed the surface tension energy $E_\sigma$ has the form
\begin{equation*}
E_\sigma = 2\sigma\lim_{x_1\rightarrow \infty}\left[\int_{x_0}^{x_1} \sqrt{1+\left(\frac{du}{dx}\right)^2}\,dx - \int_0^{x_1} \,dx\right].
\end{equation*}

We rearrange the integrals above, 
\begin{equation}\label{inteesigma}
E_\sigma = 2\sigma\lim_{x_1 \rightarrow \infty}\left[\int_{x_0}^{x_1} \left(\sqrt{1+\left(\frac{du}{dx}\right)^2}-1\right)\,dx - \int_0^{x_0}\,dx\right].
\end{equation}

Using the solution $u(\psi)$ and $x(\psi)$ in Eqs.~(\ref{solnu}) and~(\ref{solnx}), $E_{\sigma}$ can be integrated in terms of $\psi$. The parametric form works for both graph and non-graph cases. When $\psi_0>0$, Eq.~(\ref{inteesigma}) becomes 
\begin{eqnarray*}
E_{\sigma} &=& 2\sigma\int_{\psi_0}^0\left(-\sqrt{\left(\frac{du}{d\psi}\right)^2+\left(\frac{dx}{d\psi}\right)^2}+\frac{\cos\psi}{2\sqrt{\kappa}\sin\frac{\psi}{2}}\right)\,d\psi-2\sigma a\sin\phi_0\\
&=&\frac{\sigma}{\sqrt{\kappa}}\int_{0}^{\psi_0}\left(\frac{1}{\sin\frac{\psi}{2}}-\frac{\cos \psi}{\sin \frac{\psi}{2}}\right) \,d\psi -2\sigma a \sin\phi_0\\
&=& 4\frac{\sigma}{\sqrt{\kappa}}\left(1-\cos\frac{\psi_0}{2}\right) -2\sigma a \sin\phi_0.
\end{eqnarray*}

When $\psi_0<0$, Eq.~(\ref{inteesigma}) becomes
\begin{eqnarray*}
E_{\sigma} &=& 2\sigma\int_{\psi_0}^0\left(\sqrt{\left(\frac{du}{d\psi}\right)^2+\left(\frac{dx}{d\psi}\right)^2}+\frac{\cos\psi}{2\sqrt{\kappa}\sin\frac{\psi}{2}}\right)\,d\psi-2\sigma a\sin\phi_0\\
&=& \frac{\sigma}{\sqrt{\kappa}} \int_{\psi_0}^{0}\left(-\frac{1}{\sin\frac{\psi}{2}} + \frac{\cos\psi}{\sin\frac{\psi}{2}} \right) \,d\psi -2\sigma a \sin\phi_0\\
&=& 4\frac{\sigma}{\sqrt{\kappa}}\left(1-\cos\frac{\psi_0}{2}\right) -2\sigma a \sin\phi_0.
\end{eqnarray*}

Therefore, we have the surface tension energy 
\begin{equation}
E_{\sigma} = 4\frac{\sigma}{\sqrt{\kappa}}\left(1-\cos\frac{\psi_0}{2}\right) -2\sigma a \sin\phi_0.
\end{equation}

\subsection{Fluid Potential Energy $E_F$}\label{Apxb2}
In the case when the fluid interface and the cross section of the wetted region are both graphs, we break the fluid potential energy $E_F$ into two parts. 
\begin{equation}
E_F = \underbrace{2\rho g \int_{0}^{x_0} \frac{y^2}{2}\,dx}_{E_{F1}}+ \underbrace{2 \rho g\int_{x_0}^{\infty}\frac{u^2}{2}\,dx}_{E_{F2}},
\label{EFformula}
\end{equation}
where $y$ is the vertical height of the bottom of the cylinder and $u$ is the fluid height, shown in Fig.~\ref{fluidpo}, they have the form
\begin{eqnarray*}
u(\psi) &=& -\frac{2}{\sqrt{\kappa}} \sin\frac{\psi}{2},\\
y(\phi) &=& h-a\cos \phi.
\end{eqnarray*}

\begin{enumerate}
\item For $E_{F_2}$,
\begin{eqnarray*}
E_{F_2} &=& 2\rho g\int_{x_0}^{\infty} \frac{u^2}{2} \,dx = \rho g \int_{\psi_0}^{0} \left(-\frac{2}{\sqrt{\kappa}}\sin\frac{\psi}{2}\right)^2\left(-\frac{1}{2\sqrt{\kappa}}\frac{\cos\psi}{\sin\frac{\psi}{2}}\right)\,d\psi\\
&=& -\frac{2\sigma}{\sqrt{\kappa}}\int_{\psi_0}^{0} \sin\frac{\psi}{2}\cos\psi \,d\psi\\
&=& -\frac{2\sigma}{\sqrt{\kappa}}\bigg(\frac{2}{3}-\cos\frac{\psi_0}{2}+\frac{1}{3}\cos\frac{3\psi_0}{2}\bigg).
\end{eqnarray*}

With the identity $\cos\frac{3\psi_0}{2}=\cos\frac{\psi_0}{2}(2\cos\psi_0-1)$,
\begin{equation}
E_{F_2} = -\frac{4\sigma}{3\sqrt{\kappa}}\bigg(1-2\cos\frac{\psi_0}{2}+\cos\frac{\psi_0}{2}\cos\psi_0\bigg).
\end{equation}
\item For $E_{F1}$, 
\begin{eqnarray*}
E_{F_1} &=& 2\rho g \int_{0}^{x_0} \frac{y^2}{2} \,dx = \rho g \int_{0}^{\phi_0} (a\cos\phi-h)^2 a\cos\phi \,d\phi\\
&=& \rho g \int_{0}^{\phi_0}\left(a\cos\phi-a\cos\phi_0+\frac{2}{\sqrt{\kappa}}\sin\frac{\psi_0}{2}\right)^2 a\cos\phi \,d\phi\\
&=& \frac{1}{12}\rho ga^3\sin3\phi_0-\rho ga^3\phi_0\cos\phi_0+\frac{3}{4}\rho g a^3\sin\phi_0-a^2\sqrt{\sigma \rho g}\sin\frac{\psi_0}{2}\sin2\phi_0\\
&&+2a^2\sqrt{\sigma \rho g}\phi_0\sin\frac{\psi_0}{2}+4\sigma a\sin^2\frac{\psi_0}{2}\sin\phi_0.
\end{eqnarray*}
\end{enumerate}

\begin{figure}[h]
\centering
     \subfloat[]{\includegraphics[scale = 0.3]{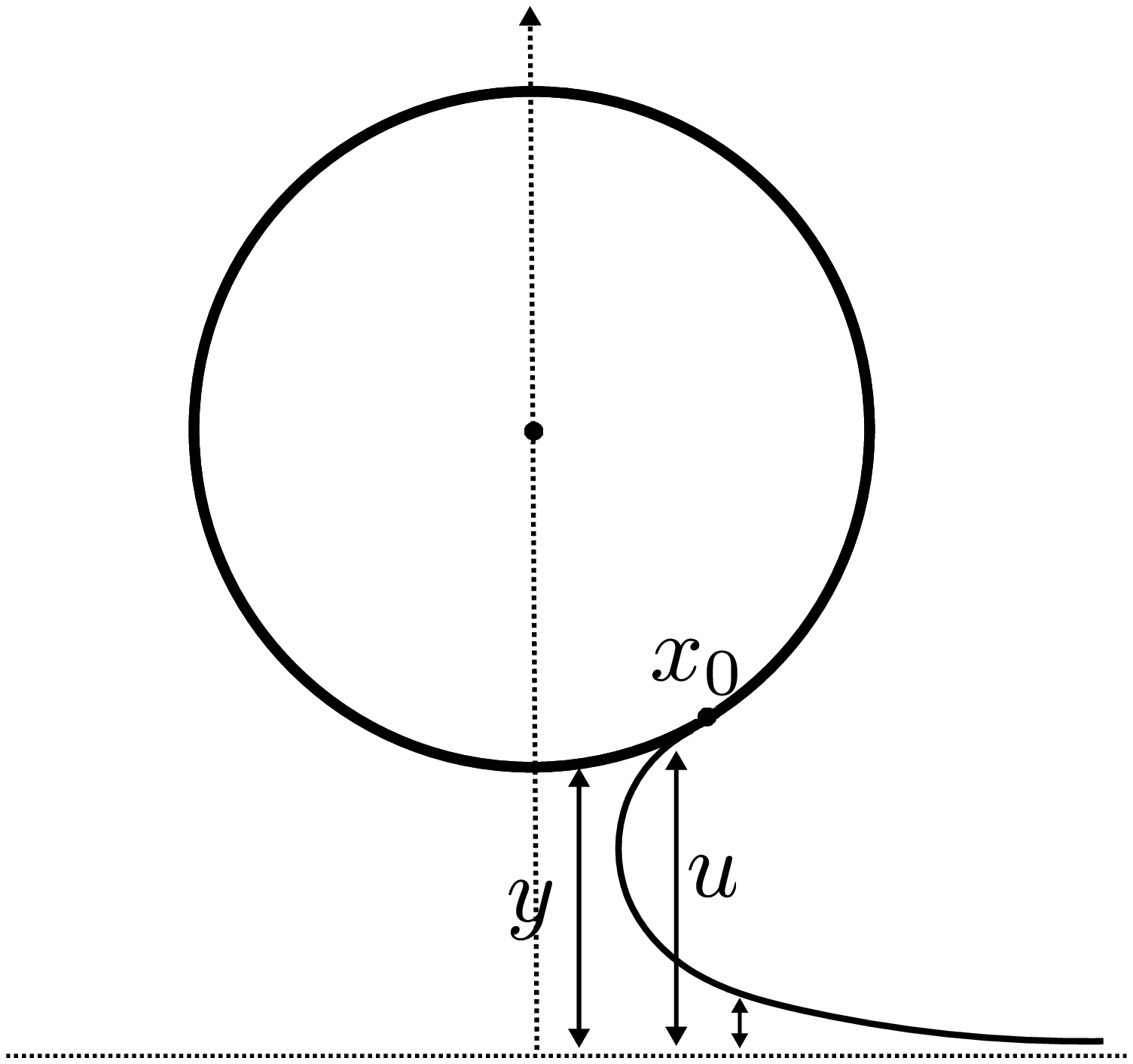}\label{fluidponongraph}}
     \hspace{1cm}
     \subfloat[]{\includegraphics[scale = 0.3]{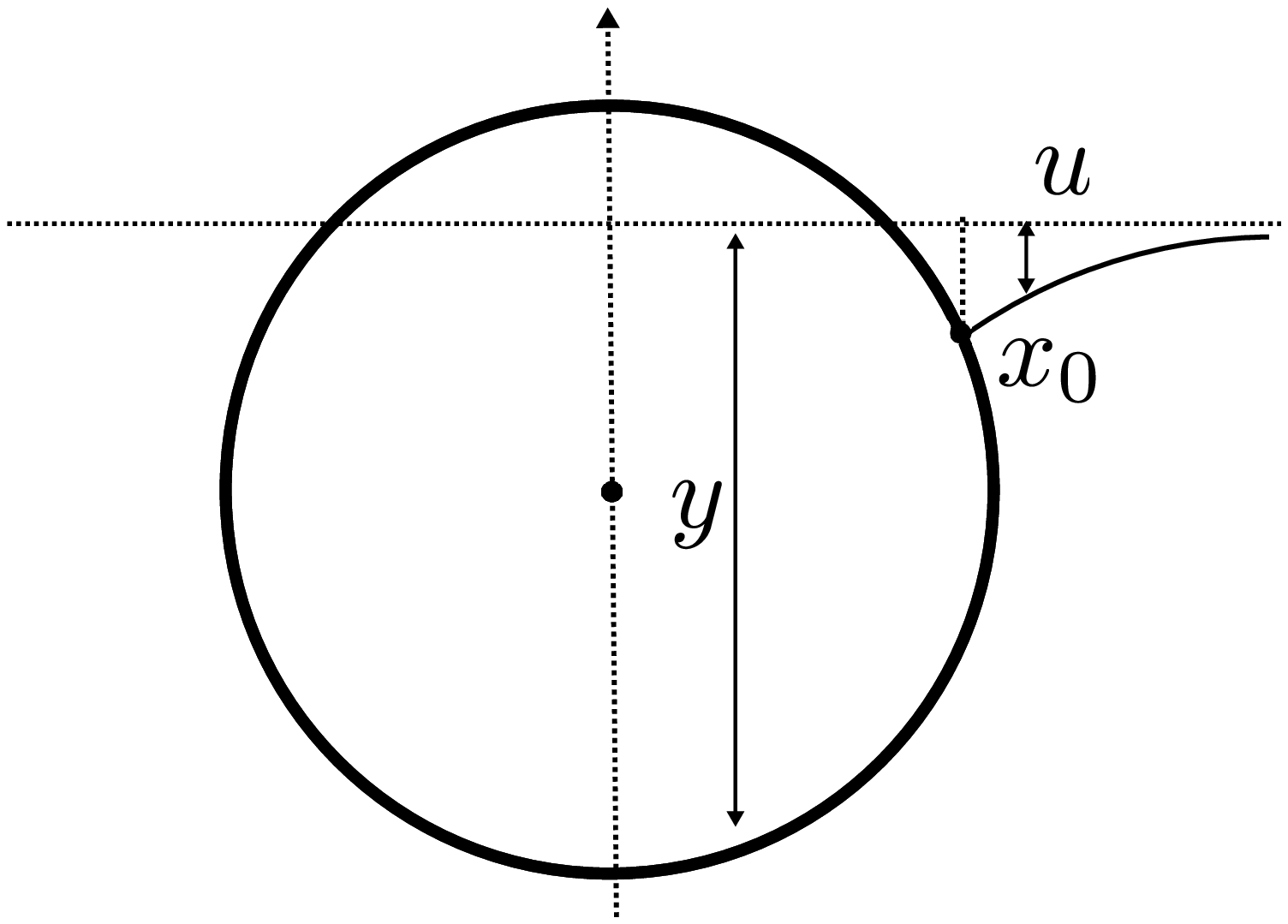}\label{wettednongraph}}
\caption{(a) The case when the fluid interface is not a graph and (b) the case when the cross section of the wetted region is a non-graph. }
\end{figure}

Thus the fluid potential energy $E_F$ has the form
\begin{eqnarray*}
E_F(\psi_0,\phi_0) = E_{F_1}+E_{F_2} &=& -\frac{4\sigma}{3\sqrt{\kappa}}\left(1-2\cos\frac{\psi_0}{2}+\cos\frac{\psi_0}{2}\cos\psi_0\right)+\frac{1}{12}\rho ga^3\sin3\phi_0\\
&&-\rho ga^3\phi_0\cos\phi_0+\frac{3}{4}\rho g a^3\sin\phi_0-a^2\sqrt{\sigma \rho g}\sin\frac{\psi_0}{2}\sin2\phi_0\\
&&+2a^2\sqrt{\sigma \rho g}\phi_0\sin\frac{\psi_0}{2}+4\sigma a\sin^2\frac{\psi_0}{2}\sin\phi_0.
\end{eqnarray*}

Moreover, if the fluid interface or the cross section of the wetted region is not a graph (see Fig.~\ref{fluidponongraph} and Fig.~\ref{wettednongraph}), the expression of $E_F(\psi_0,\phi_0)$ also holds. But in addition, we have to assume that there is no intersection of the fluid interfaces.

\section{Analysis of the Buoyant Force}\label{Apx2}
In this section, we will examine how Archimedes' principle works in the no surface tension case and another way to approach buoyant force using the divergence theorem. 

The buoyant force has the form:
\begin{equation}
F_B = \hat{k} \cdot \int\limits_{\Sigma} \vec{F}\,ds,
\end{equation}
where the centripetal component pressure $\vec{F} = \rho g y \hat{n}_c$, $\hat{n}_c$ is the outer unit normal of the cylinder, $\hat{k}$ is the unit vertical vector pointing upward and $\Sigma$ is the wetted region of the cylinder.  

\begin{figure}[h]
\begin{center}
\includegraphics[scale = 0.4]{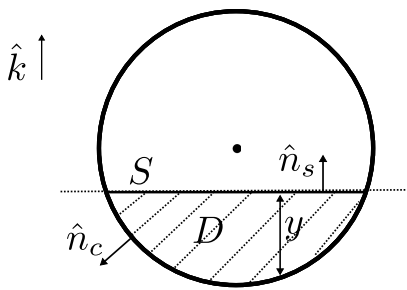}
\end{center}
\caption{Archimedes' principle.}
\label{archimedes}
\end{figure}

\begin{enumerate}
\item If no surface tension exists (shown in Fig.~\ref{archimedes}), 
\begin{eqnarray*}
F_B &=& \hat{k} \cdot \int\limits_{\Sigma} \vec{F}\,ds = \int\limits_{\Sigma} \left(\rho g y \hat{k}\right)\cdot \hat{n} \,ds + \underbrace{ \int\limits_{S}\left(\rho g y \hat{k}\right)\cdot \hat{n} \,ds }_{=0,\textrm{ since } y=0 \textrm{ on } S.} \\
&=& \int\limits_{\Sigma\cup S} \left(\rho g y \hat{k}\right) \cdot \hat{n} \,ds = \int_{D} \rho g \,dA\\
&=& \rho g \left\vert D\right\vert,
\end{eqnarray*}
where $\partial D = \Sigma\cup S$ and $\hat{n} \in \{\hat{n}_c,\hat{n}_s\}$ is the outer normal of $\partial D$. When the divergence theorem is applied, $F_B = \rho g |D|$, which is known as Archimedes' principle. 

\item If surface tension is present (shown in Fig. \ref{bfst}), 
\begin{eqnarray*}
F_B &=& \hat{k} \cdot \int\limits_{\Sigma} \vec{F}\,ds\\
&=& \int\limits_{\Sigma} \left(\rho g y \hat{k}\right)\cdot \hat{n}_c \,ds + \underbrace{ \int\limits_{S_{top}}\left(\rho g y \hat{k}\right)\cdot \hat{n}_s \,ds }_{=0, \textrm{ since } y=0 \textrm{ on } S_{top}.} +\underbrace{ \int\limits_{S_{side}}\left(\rho g y \hat{k}\right)\cdot \hat{n}_s \,ds }_{=0, \textrm{ since } \hat{k}\cdot \hat{n}_s=0.} \\
&=& \int\limits_{\Sigma\cup S_{top} \cup S_{side}} \left(\rho g y \hat{k}\right) \cdot \hat{n} \,ds = \int\limits_{D} \rho g \,dA\\
&=& \rho g\left\vert D\right\vert,
\end{eqnarray*}

where $\partial D = \Sigma\cup S_{top} \cup S_{side}$ and $\hat{n} \in \{\hat{n}_c,\hat{n}_s\}$ is the outer normal of $\partial D$. When the divergence theorem is applied, $F_B = \rho g \left\vert D\right\vert$. But Archimedes' principle doesn't hold anymore. The enclosed area is no longer the immersed region due to the presence of surface tension. 
\end{enumerate}

\begin{figure}
\begin{center}
\includegraphics[scale = 0.4]{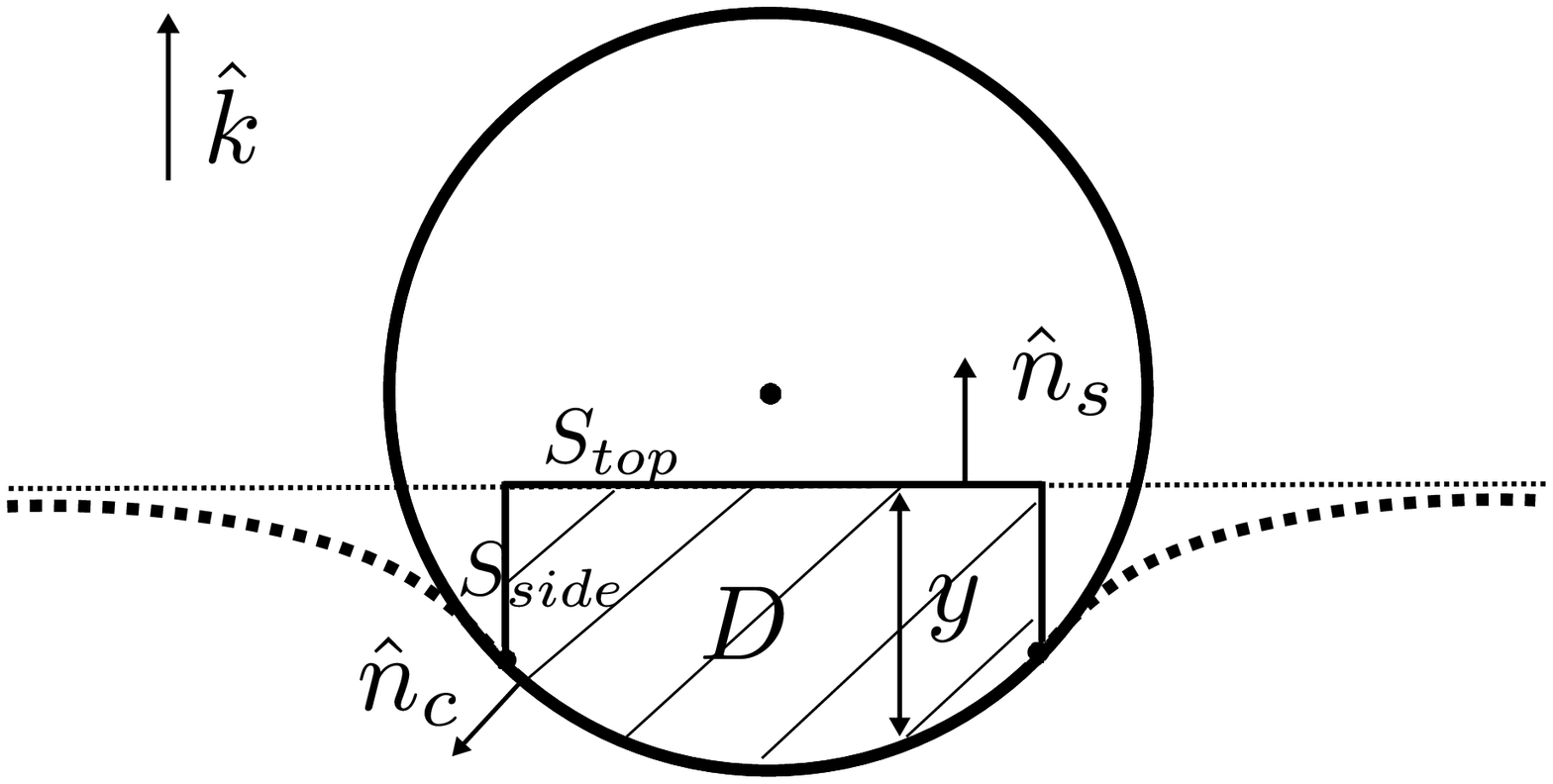}
\hspace{0.5cm}
\includegraphics[scale = 0.4]{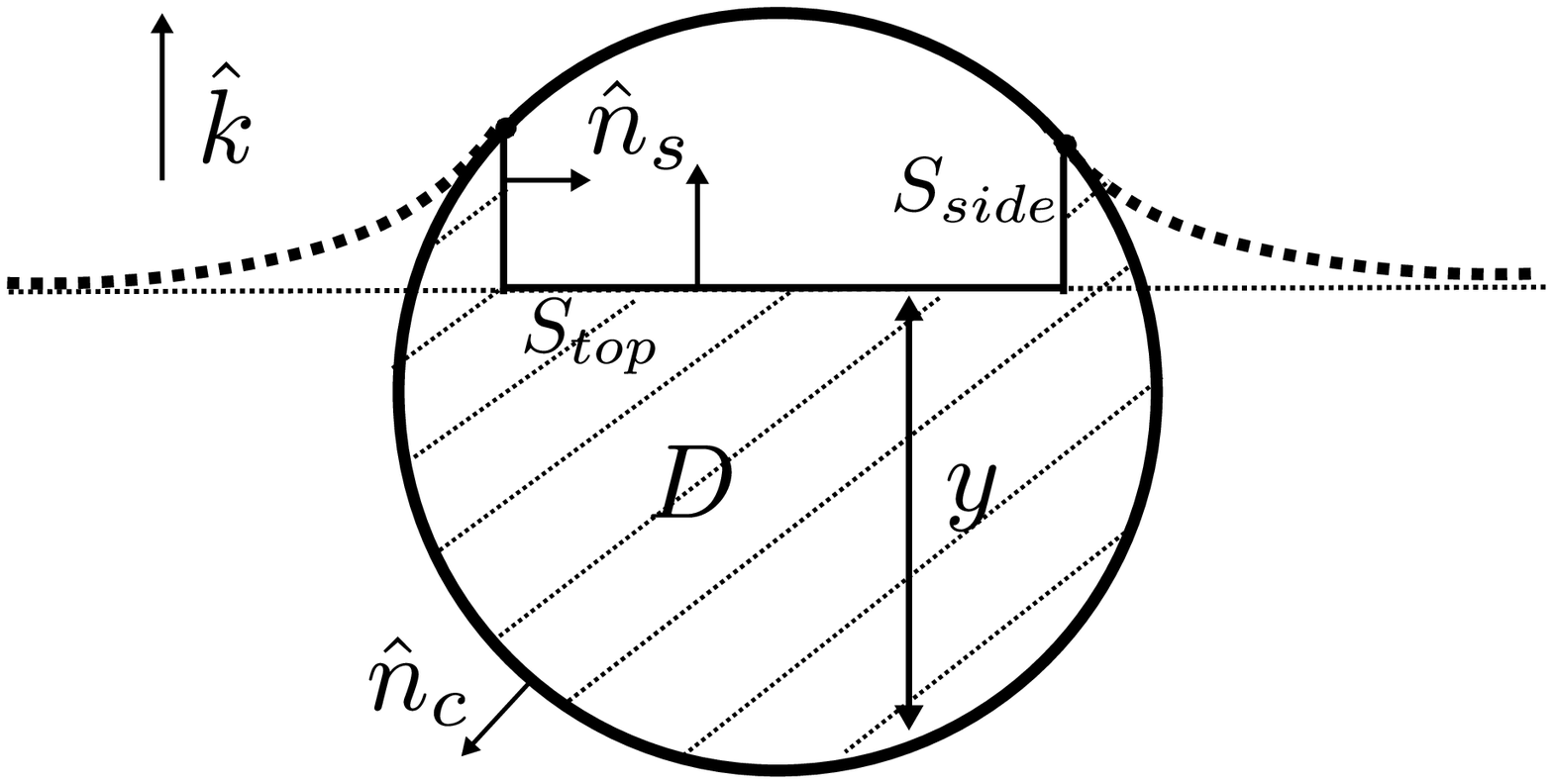}
\end{center}
\caption{The buoyant force when surface tension is present.}
\label{bfst}
\end{figure}

\section{Relation between $\frac{dE_T}{dh}$ and $F_T$}\label{Apxc}
In this section, we'll derive the relation between $E_T$ and $F_T$. First, we take derivative of $E_T$ in terms of $\phi_0$,
\begin{eqnarray*}
\frac{dE_T}{d\phi_0}(\phi_0) &=& -mga\sin\phi_0-mg\sqrt{\frac{\sigma}{\rho g}}\sin\left(\frac{\phi_0+\gamma}{2}\right)\\
&&-2\sigma a\sin\phi_0\sin(\phi_0+\gamma)-2\sigma \sqrt{\frac{\sigma}{\rho g}}\sin(\phi_0+\gamma)\sin\left(\frac{\phi_0+\gamma}{2}\right)\\
&&-4\sigma a \sin\left(\frac{\phi_0+\gamma}{2}\right)\cos\left(\frac{\phi_0+\gamma}{2}\right)\sin\phi_0-4a^2\sqrt{\sigma \rho g}\sin^2\phi_0\cos\left(\frac{\phi_0+\gamma}{2}\right)\\
&&-\frac{1}{2}\rho g a^3\sin\phi_0\sin2\phi_0-\frac{1}{2}a^2\sqrt{\sigma \rho g}\sin\left(\frac{\phi_0+\gamma}{2}\right)\sin2\phi_0\\
&&+\rho ga^3\phi_0\sin\phi_0+\sqrt{\sigma \rho g}a^2\phi_0\sin\left(\frac{\phi_0+\gamma}{2}\right).
\end{eqnarray*}
After arranging, we factor the common term
\begin{equation}
a\sin\phi_0+\sqrt{\frac{\sigma}{\rho g}}\sin\left(\frac{\phi_0+\gamma}{2}\right).
\end{equation}

\begin{eqnarray*}
\frac{dE_T}{d\phi_0}(\phi_0)&=&-mg\left[a\sin\phi_0+\sqrt{\frac{\sigma}{\rho g}}\sin\left(\frac{\phi_0+\gamma}{2}\right)\right]\\
&&-2\sigma \sin(\phi_0+\gamma)\left[a\sin\phi_0+\sqrt{\frac{\sigma}{\rho g}}\sin\left(\frac{\phi_0+\gamma}{2}\right)\right]\\
&&-4a\sqrt{\sigma \rho g}\cos\left(\frac{\phi_0+\gamma}{2}\right)\sin\phi_0\left[a\sin\phi_0+\sqrt{\frac{\sigma}{\rho g}}\sin\left(\frac{\phi_0+\gamma}{2}\right)\right]\\
&&-\frac{1}{2}\rho g a^2\sin2\phi_0\left[a\sin\phi_0+\sqrt{\frac{\sigma}{\rho g}}\sin\left(\frac{\phi_0+\gamma}{2}\right)\right]\\
&&+\rho g a^2\phi_0\left[a\sin\phi_0+\sqrt{\frac{\sigma}{\rho g}}\sin\left(\frac{\phi_0+\gamma}{2}\right)\right].
\end{eqnarray*}

We multiple the term $\frac{d\phi_0}{dh}$ (since $\frac{dh}{d\phi_0}<0$ on $\phi_0 \in \left[0,\pi\right]$ except $\phi_0 =\gamma = 0$ and $\phi_0 = \gamma = \pi$), equivalently, the chain rule is applied.
\begin{eqnarray*}
-\frac{dE_T}{d\phi_0}\frac{d\phi_0}{d h}&=&-\bigg[-mg-2\sigma \sin(\phi_0+\gamma)-4a\sqrt{\sigma \rho g}\cos\left(\frac{\phi_0+\gamma}{2}\right)\sin\phi_0\\
&& -\frac{1}{2}\rho g a^2\sin2\phi_0+\rho g a^2\phi_0\bigg]\left[a\sin\phi_0+\sqrt{\frac{\sigma}{\rho g}}\sin\left(\frac{\phi_0+\gamma}{2}\right)\right]\\
&& \left[-\frac{1}{a\sin\phi_0+\sqrt{\frac{\sigma}{\rho g}}\sin\left(\frac{\phi_0+\gamma}{2}\right)}\right]\\
&=& -mg-2\sigma \sin(\phi_0+\gamma)-4a\sqrt{\sigma \rho g}\cos\left(\frac{\phi_0+\gamma}{2}\right)\sin\phi_0\\
&& -\frac{1}{2}\rho g a^2\sin2\phi_0+\rho g a^2\phi_0\\ 
&=& F_T.
\end{eqnarray*}

Therefore, we obtain the relation
\begin{equation}
-\frac{dE_T}{dh} = F_T.
\end{equation}


%
%

%


\nocite{*}
\bibliography{floating}

\end{document}